\documentclass[]{article}
\usepackage{amsmath}
\usepackage{amssymb}
\usepackage{amsfonts}
\usepackage{amsthm}
\usepackage{moreverb}
\usepackage[disable]{todonotes}
\usepackage{subfig}
\newcommand\BibTeX{{\rmfamily B\kern-.05em \textsc{i\kern-.025em b}\kern-.08em
T\kern-.1667em\lower.7ex\hbox{E}\kern-.125emX}}

\newtheorem{thm}{Theorem}[section]
\newtheorem{cor}[thm]{Corollary}

\newtheorem{prop}[thm]{Proposition}

\theoremstyle{definition}

\numberwithin{equation}{section}

\renewcommand{\i}{\imath}
\renewcommand{\vec}{\mathbf}

\providecommand{\br}{}

\newcommand{\hcont}{h_{\gamma}} 
\newcommand{\thetacont}{\theta_{\gamma}} 
                                                                   
\newcommand{\thetaRz}{\theta_{\alpha}} 
                                                               
\newcommand{\thetacsg}{\theta_{\beta}}  

\title{Analyzing the wave number dependency of the convergence rate of
  a multigrid preconditioned Krylov method for the Helmholtz equation with an absorbing layer}
\author{B. Reps and W. Vanroose\footnote{Email: {\tt wim.vanroose@ua.ac.be}}}
\date{October 2011}

\begin{document}

\maketitle
\begin{center}
{\it Dept. Mathematics and Computer Science, Universiteit Antwerpen, Middelheimlaan 1, 2020
  Antwerpen, Belgium}
\end{center}
\begin{abstract}
  This paper analyzes the Krylov convergence rate of a Helmholtz
  problem preconditioned with Multigrid.  The multigrid \br{method} is applied to
  the Helmholtz problem formulated on a complex contour and uses
  \br{GMRES as a smoother substitute} at each level.  A one-dimensional model
  is analyzed both in a continuous and discrete way.  It is shown that
  the Krylov convergence rate of the continuous problem is independent
  of the wave number.  The discrete problem, however, can deviate
  significantly from this bound due to a pitchfork in the spectrum.  It is further
  shown in numerical experiments that the convergence rate of the
  Krylov method approaches the continuous bound as the grid distance $h$ gets small. 
\end{abstract}


\vspace{-6pt}
\section{Introduction}
The Helmholtz equation plays a central role in seismic imaging,
electromagnetic scattering and many other applications.  For $x \in
\Omega \subset \mathbb{R}^d$ the equation reads
\begin{equation}
  Hu\left(\vec{x}\right) \equiv \left[-\triangle - k\left(\vec{x}\right)^2\right] u\left(\vec{x}\right) = f(\vec{x}) \,,
\end{equation}
where the wave number $k(\vec{x})$ depends on the coordinates $\vec{x}$ and can model, for
example, the change in refractive index of the material through which
electromagnetic waves are propagating, $f(\vec{x})$ models the source of the
waves and $-\triangle$ is the $d$-dimensional negative Laplacian.

\br{In theory many applications need to be solved on an infinite domain, yet in practice a numerical solution method must truncate the domain in some way. Therefore, on the finite computational domain the equation is solved with outgoing wave conditions on the artificially introduced boundaries}.  Over the years
many good outgoing wave boundary conditions have been proposed such as \textit{Perfectly Matched Layers} (PML) \cite{B94,CW94}. \br{This leads to a spectrum of the operator where each eigenvalue has an imaginary part to represent the damping of the outgoing waves on the exterior layers}.  In a similar
way in the physics and chemistry literature absorbing boundary
conditions based on \emph{Exterior Complex Scaling} (ECS) \cite{AC71,BC71,S79} are used for example
in break-up problems \cite{mccurdy2004solving}.

After discretization the Helmholtz problem becomes a linear system
$H_hu_h=f_h$.  Due to the \br{negative shift, with a magnitude that depends on the wave number $k$,} the matrix $A$ is indefinite. Indeed,
the wave number shifts the spectrum of $-\triangle$, which is positive
definite, to the left. The eigenvalues corresponding to the smooth
modes can get close to zero or have a negative real part. These
spectral properties lead to a large condition number and iterative
methods perform poorly. Efficient preconditioners for the negative
Laplacian, such as a multigrid method, fail when applied to the Helmholtz
problem.  A recent review of the difficulties of iterative methods for the
Helmholtz problem is given in \cite{gander2010}.

\br{An important step in the improvement of the iterative solution of the Helmholtz problem was taken by Bayliss, Goldstein and Turkel \cite{BGT83,BGT85} with the shifted Laplacian preconditioner. Instead of approximately inverting the original Helmholtz operator with e.g.\ ILU or a few multigrid cycles as a preconditioning step, the Laplacian or positively shifted Laplacian is used as a preconditioner. This positive definite operator serves as an approximation of the Helmholtz operator and can be efficiently inverted with standard iterative methods. A significant extension of this idea was made by the introduction of the \emph{complex shifted Laplacian}, a related Helmholtz problem with a complex-valued wave number, which makes a better preconditioner and can still be efficiently solved with multigrid \cite{EVO04,EVO06}.}  The complex-valued wave number prevents that eigenvalues of the preconditioner come close to zero at any level of the multigrid hierarchy. This is particularly useful in the coarse grid correction where diverging \br{numerical} resonances can appear when an eigenvalue of a coarse level approaches the origin \cite{Elman}.

In \cite{JCP-paper} it was shown that scaling the wave number with a complex value has the same effect as scaling the grid distance with a complex value. As a result the Helmholtz problem can be efficiently preconditioned by a Helmholtz operator discretized on a complex-valued grid.  This might be of interest for problems where complex-valued grid distances are already used to implement the absorbing boundary conditions.  This is the case for ECS \cite{mccurdy2004solving} or PML \cite{CW94}.

The introduction of complex wave numbers (or grids) avoids the appearance of resonances, however, it does not prevent traditional smoothers like $\omega$-Jacobi or Gauss-Seidel to be unstable for the smooth modes. \br{In this paper we are interested in developing a matrix-free method, though we mention that also ILU smoothers can be unstable for similar reasons. In \cite{polynomialsmoother} we analyze GMRES as a replacement of the traditional smoothers when multigrid is applied to a preconditioning operator based on complex-valued grids. Numerical experiments show that only a few GMRES iterations are needed at every level, as opposed to the results in \cite{Elman} where multigrid was used to invert the original Helmholtz operator which requires more GMRES iterations at some intermediate levels.}

Note that for the complex shifted Laplacian preconditioner the complex shift has a parameter. The choice of the parameter is analyzed in \cite{vanGijzen2007spectral} and \cite{osei2010preconditioning}.

We mention other promising preconditioning techniques such as Moving
Perfectly Matched Layers \cite{mpml}, a transformation that turns the
Helmholtz problem into a reaction-advection-diffusion problem
\cite{HM11}, application of separation-of-variables \cite{plessix},
algebraic multilevel methods \cite{boll}, the wave-ray method
\cite{brandt1997wave,livshits2006accuracy}, and combined complex
shifted Laplacian and deflation \cite{sheikh2009fast}.

\br{In this paper we focus the analysis on the wave number dependency of the convergence behavior of a preconditioned Krylov subspace method.} The preconditioning \br{operator} is the Helmholtz operator discretized on a complex-valued grid and it is inverted with multigrid using GMRES as a smoother substitute as suggested in \cite{Elman} and \cite{polynomialsmoother}. The paper starts with a review of a one-dimensional continuous model problem in Section~\ref{sec:modelproblem}. For this problem the eigenvalues of the preconditioned \br{operator} are derived analytically and we find that the Krylov convergence rate should be independent of the wave number in Section~\ref{sec:continuouseigenvaluesprecon}. The discrete problem, discussed in Section~\ref{sec:discrete}, however, does not have this bound. We explain the origin of this deviation and give estimates for the different regions in the convergence as a function of the wave number $k$.  In Section~\ref{sec:numerical} we illustrate the theory with numerical examples.

%
%
\section{Model problem}\label{sec:modelproblem}
In this section we formulate a one-dimensional Helmholtz model problem
that is representative for higher dimensional problems that arise in many applications. It is a Helmholtz problem with a constant wave number
$k$ on the domain $\Omega=[0,1]\in\mathbb{R}$,
\begin{equation}\label{eq:helm0}
\begin{cases}
  Hu(x)\equiv \left[-\frac{d^2}{dx^2} -k^2\right]u(x) &= f(x), \quad \forall x \in (0,1);\\
  u(0) = 0;\\
  u(1) = \mbox{outgoing wave,}
\end{cases}
\end{equation} 
with a zero Dirichlet boundary condition on the left boundary $x=0$ and an outgoing wave boundary condition on the right boundary $x=1$. The right hand side $f(x)$ represents a localized source term.

\subsection{The Helmholtz problem with ECS}
The outgoing wave boundary condition in \eqref{eq:helm0} is implemented with exterior complex scaling (ECS) \cite{mccurdy2004solving}, an equivalent formulation of the PML technique by B\'erenger \cite{B94}. Therefore the domain is extended to $\Omega\cup\Gamma=[0,1]\cup(1,R]\in\mathbb{R}$ after which a complex coordinate transformation is defined as,
\begin{equation}\label{eq:domain}
 z(x) = \left\{
  \begin{array}{ll}
    x, & \hbox{$x \in [0,1]$;} \\
    1+(x-1)e^{\imath\thetacont}, & \hbox{$x \in (1,R]$.}
  \end{array}
\right.
\end{equation}
We write $R_z=z(R)\in\mathbb{C}$ for the new complex right boundary. This results in the domain $\Omega\cup\Gamma_z=[0,1]\cup(1,R_z]\in\mathbb{C}$ that is the union of the original real domain $[0,1]$ and a complex line connecting the point $1$ to $R_z$, see Figure~\ref{fig:domain}. In this paper we use linear complex scaling by simply rotating the absorbing layer over an angle $\thetacont$ in the complex plane, but smoother transitions, with a non-constant angle, are also possible. Posing a zero Dirichlet boundary condition in $R_z$ implies an outgoing wave in the original right boundary $x=1$ \cite{CW94}.

The Helmholtz problem \eqref{eq:helm0} translates into
\begin{equation}\label{eq:helm}
\begin{cases}
  Hu(z)\equiv \left[-\frac{d^2}{dz^2} -k^2\right]u(z) &= f(z), \quad \forall z \in (0,1]\cup(1,R_z];\\
  u(0) = u(R_z) = 0,\\
\end{cases}
\end{equation} with homogeneous Dirichlet boundary conditions at $z(0)=0$ and $z(R)=R_z$. Note that the source term $f(z)$ was assumed to vanish outside $[0,1]$.

We define the ECS grid on the complex stretched domain \eqref{eq:domain},
\begin{equation}\label{eq:ecsgrid}                                                                                                      
(z_j)_{0\leq j\leq n+m} =                                                                                                               
\begin{cases}                                                                                                                           
        j h,  & (0\leq j\leq n);\\
        1+(j-n)\hcont, & (n+1\leq j\leq n+m),
\end{cases}                                                                                                                             
\end{equation}
that consists of $n$ intervals of the grid distance \br{$h=1/n$} followed by
$m$ intervals of the complex grid distance \br{$\hcont=(R-1)e^{\i\thetacont}/m$} for the complex contour as illustrated in
Figure~\ref{fig:domain}. We discretize the second derivative operator on the grid \eqref{eq:ecsgrid} with
the Shortley-Weller finite difference scheme for non-uniform grids
\begin{equation*}\label{eq:shortwell}
\frac{d^2u}{dz^2}(z_j) \approx
\frac{2}{h_{j-1}+h_j}\left(\frac{1}{h_{j-1}}u_{j-1}-\left(\frac{1}{h_{j-1}}+\frac{1}{h_j}\right)u_j    
+\frac{1}{h_j}u_{j+1}\right)                                                                        
\end{equation*}
in grid point $j$, where $h_{j-1}$ and $h_j$ are the left and right grid distance respectively, and may belong either to the $h$ category or
to the $\hcont$ category. The result is a linear system of equations
\begin{equation}\label{eq:discretehelm}
H_hu_h \equiv (-L_h-k^2I_h)u_h = f_h,
\end{equation}
with a unique solution $u_h$ that approximates the continuous solution $u$ of the Helmholtz equation \eqref{eq:helm}. The higher dimensional Laplacian $\triangle$ is then constructed with Kronecker products of this one-dimensional discrete Laplacian matrix $L_h$.
\begin{figure}[h!]
\includegraphics[width=0.9\textwidth]{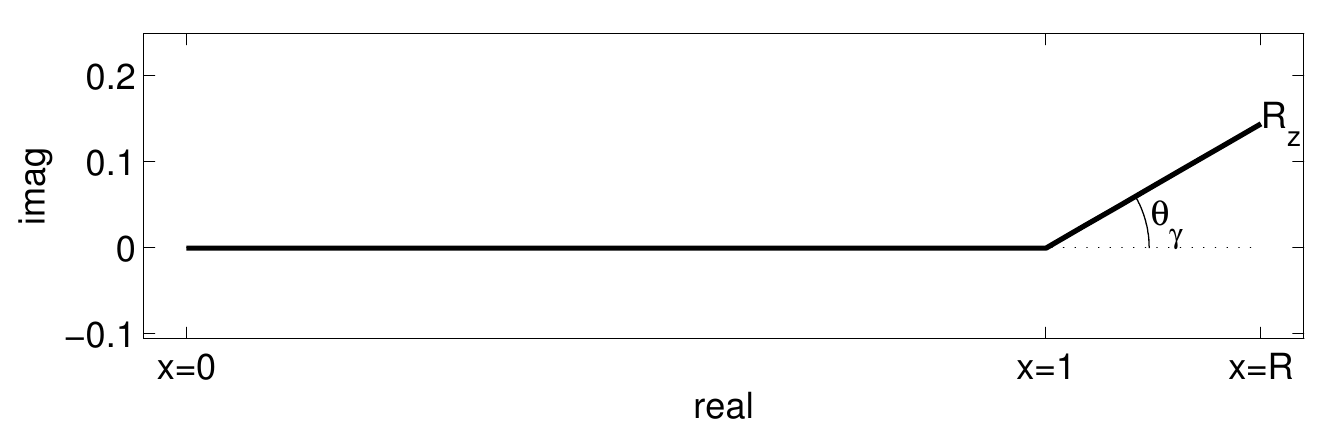}\\
\includegraphics[width=0.9\textwidth]{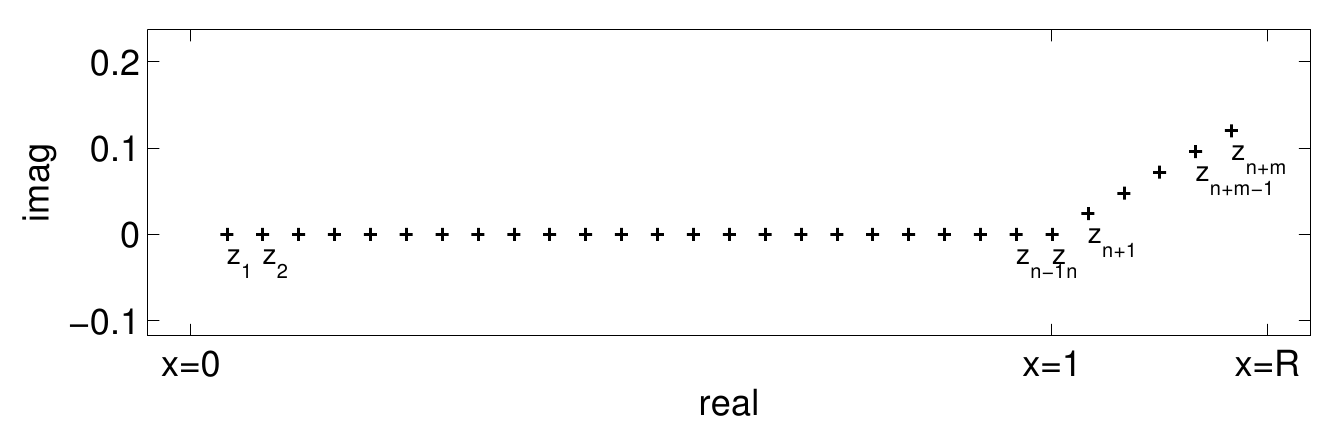}
\caption{The domain of the model problem \eqref{eq:helm} on the continuous domain (top) and the discrete grid (bottom). \label{fig:domain}}
\end{figure}

\subsection{Spectrum of the discretization matrix}\label{ssec:specdiscretehelm}
For the one-dimensional model problem the solution $u_h$ of \eqref{eq:discretehelm} is easily found with an exact inversion of the tridiagonal matrix $H_h$ in Equation~\eqref{eq:discretehelm}. As the bandwidth of the matrix grows with the dimension of the problem, so does the computational cost of direct methods. Iterative methods need to be used instead, such as multigrid and Krylov subspace solvers. The one-dimensional model has been analyzed in \cite{JCP-paper} in order to help in configuring these methods efficiently. More specifically their performance depends on the position of the eigenvalues of the matrix $H_h$ in the complex plane. Define $\gamma = \frac{h_\gamma}{h}$, then the eigenvalues of $-L_h$ are the solutions of
\begin{equation}\label{eq:eigcond}
F(t) \equiv \frac{\tan(2n p(t))}{\tan(2m q(t))}+\frac{\cos(p(t))}{\cos(q(t))} = 0,
\end{equation}
with $p(t)=\frac{1}{2}\arccos(1-\frac{t}{2}h^2)$, $q(t)=\frac{1}{2}\arccos(1-\frac{t}{2}\gamma^2
h^2)$.\\
Figure~\ref{fig:pitchfork} shows that the spectrum ($\bullet$) of $-L_h$ has a typical pitchfork shape. It is bounded in the complex plane by a triangle $\widehat{t_0t_1t_2}$ described by the points $t_0=0$, $t_1=4/h^2$ and $t_2=4/\hcont^2$. Starting in the origin $t_0=0$ we find eigenvalues along the complex line $\rho e^{-2\imath\thetaRz}$ with $\rho>0$, where $\thetaRz$ is the argument of $R_z$. \br{It was shown in \cite{JCP-paper} that these} eigenvalues approximate the smallest eigenvalues ($\times$) of the continuous Laplacian operator $-\triangle$ that will be derived in Section~\ref{sec:continuouseigenvalues}. They correspond to the smoothest modes spread over the entire ECS domain \br{and we will therefore call them the \emph{smooth} eigenvalues}. At a certain point $t_b$ the line of smooth eigenvalues splits up into two branches. One pronounced complex branch \br{consists of eigenvalues with associated eigenvectors that have their largest components at indices $n\leq j \leq n+m$. Since these eigenvectors have nearly-zero components at indices $1\leq j \leq n-1$ that correspond to the interior real region of the grid in \eqref{eq:ecsgrid}, we say that they are mainly located on the complex contour of the domain $\Gamma_z=[1,R_z]$. Whereas the other branch of eigenvalues in the spectrum lies closer to the real axis and corresponds to eigenvectors with their largest components at indices $1\leq j \leq n-1$, in other words, they are located on the real interior domain $\Omega=[0,1]$, see also Figure~\ref{fig:outereigvec}.} Together with the line of smooth eigenvalues the latter branch causes potential numerical problems as they lie close to the real axis around the points $t_0=0$ and $t_1=4/h^2$. For the entire Helmholtz operator $H_h$ in \eqref{eq:discretehelm} with a constant wave number $k$ the pitchfork shaped spectrum, and the bounding triangle, is shifted in the negative real direction over a distance $k^2$.

\begin{figure}[h!]
\begin{center}
  \includegraphics[width=\textwidth]{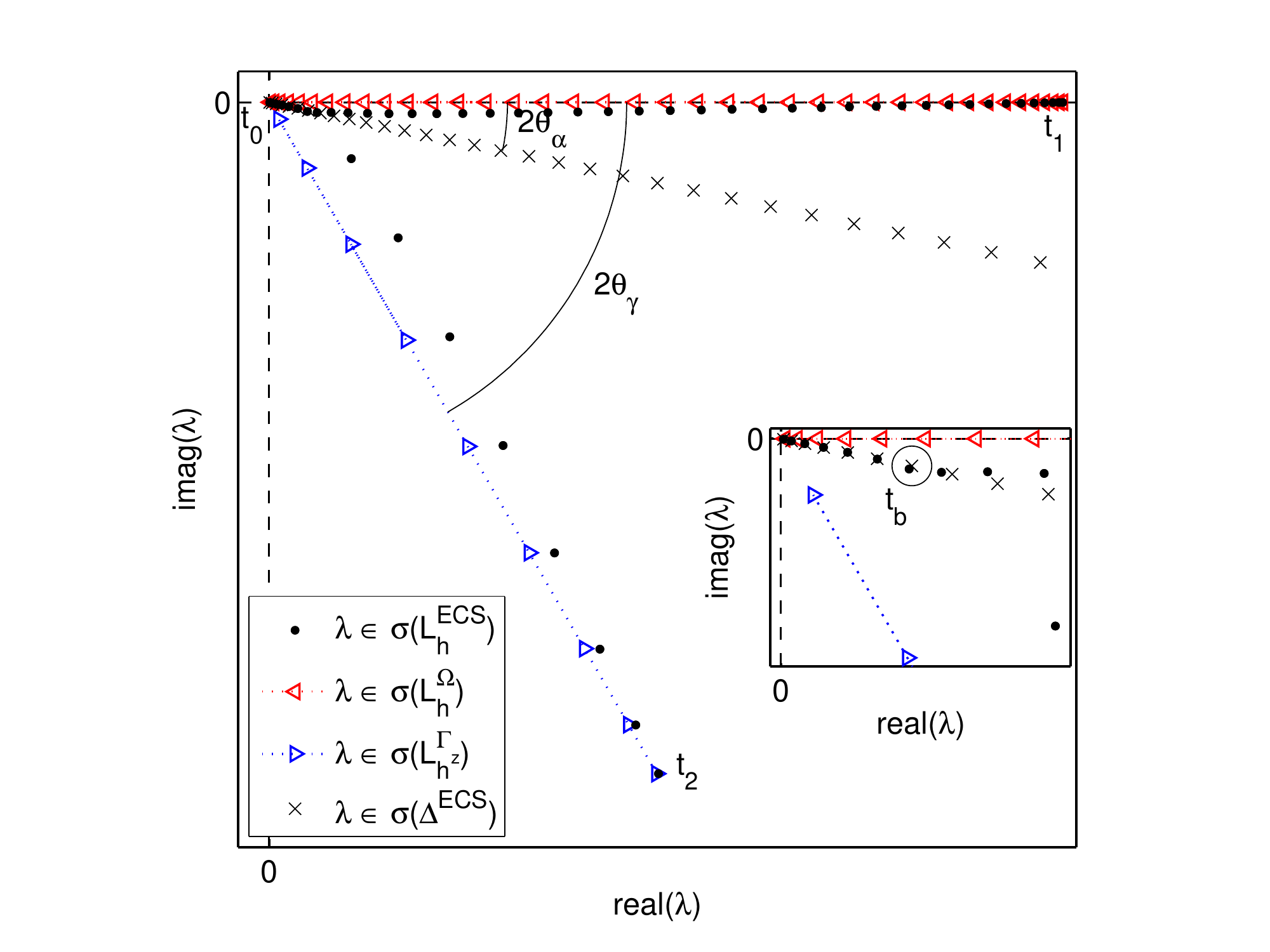}\caption{The eigenvalues of the Laplacian discretized on the ECS domain ($\bullet$) lie along a pitchfork shaped figure in the lower half of the complex plane, close to the eigenvalues of the Laplacian restricted to the interior real domain ($\triangleleft$) and the complex contour ($\triangleright$) respectively. The smallest eigenvalues approximate the eigenvalues of the continuous Laplacian ($\times$), until they split up in a point $t_b$ ({\Large$\circ$}), into two branches with limiting points $t_1 = 4/h^2$ and $t_2 = 4/(\hcont)^2$. Eigenvalues accumulate near the real axis around $0$ and $t_1$.}
\label{fig:pitchfork}
\end{center}
\end{figure}

\br{
\begin{figure}
\begin{center}
\includegraphics[width = \textwidth]{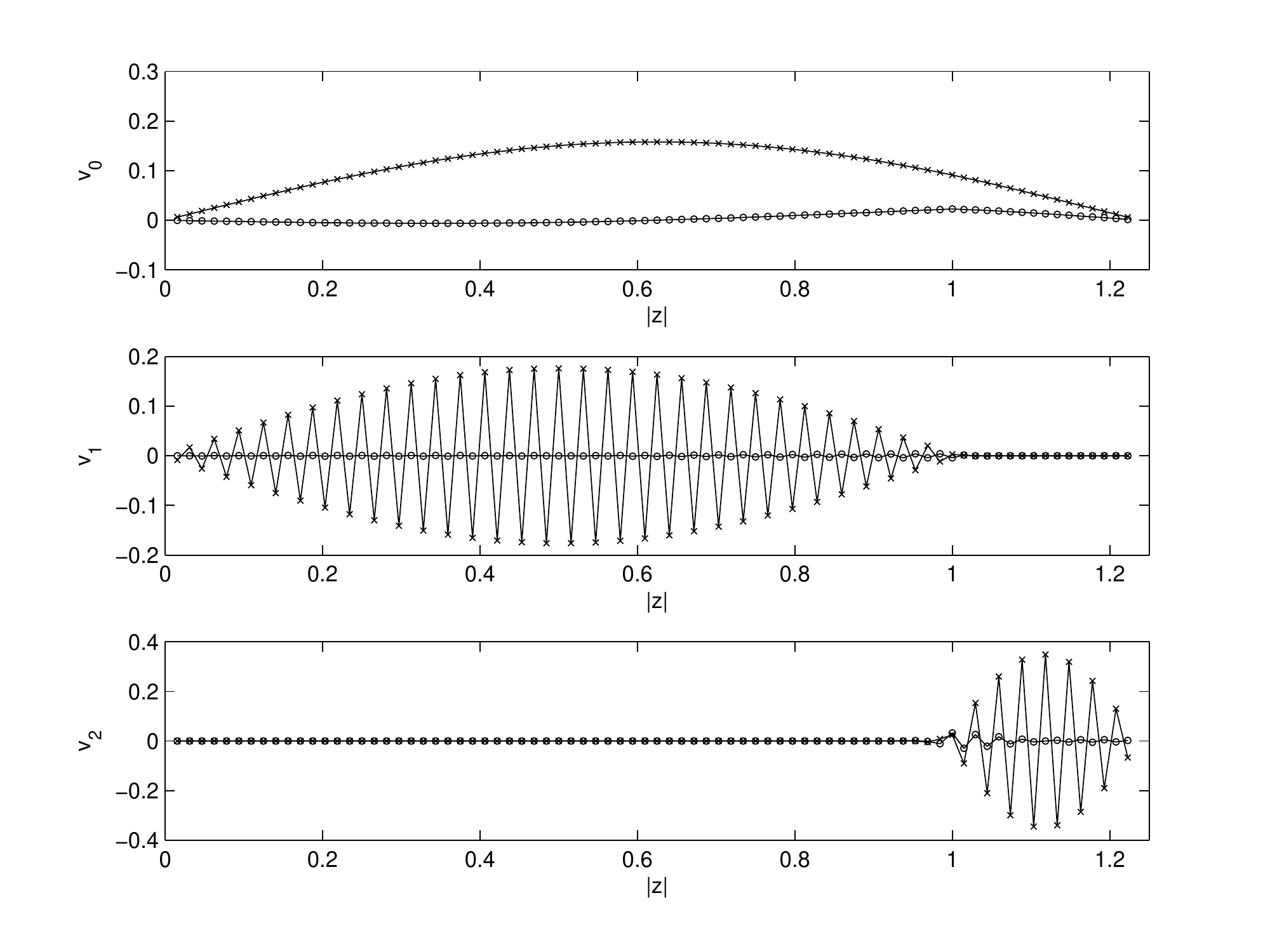}
\end{center}
\caption[]{Real part ($\times$) and imaginary part ($\circ$) of the three eigenvectors $v_0,v_1$ and $v_2$ associated to the extreme eigenvalues $\lambda_0\approx t_0$, $\lambda_1 \approx t_1$ and $\lambda_2 \approx t_2$ of the pitchfork shaped spectrum in Figure~\ref{fig:pitchfork}. Eigenvector $v_0$ (top) is the smoothest and is stretched over the entire domain, interior and exterior ECS contour. Eigenvector $v_1$ (middle) and $v_2$ (bottom) are highly oscillatory and mainly localized on the interior and exterior region respectively of the ECS domain.}
\label{fig:outereigvec}
\end{figure}
}

As the higher dimensional Helmholtz problems are constructed with Kronecker products, these results on the spectrum of the discretization matrix $H_h$ are easily extended. Every eigenvalue $\lambda$ of the $d$-dimensional Laplacian is a sum of eigenvalues $\lambda^{(l)}$ of the one-dimensional cases, $\lambda = \sum_{l=1}^{d}{\lambda^{(l)}}$. This allows us to stick the discussion to the basic academic one-dimensional model problem. Note that real applications may require more carefully engineered domains with e.g. smoother complex stretching, higher order discretization methods or an absorbing ECS layer on both sides of the domain. These generalizations might have an effect on the eigenvalues of the discretization matrix, but the main topology remains a bounded pitchfork shaped spectrum with the smoothest eigenvalues aligned, close to the continuous case.

If we assume that the spectrum of the Helmholtz discretization matrix lies inside the triangle $\widehat{t_0t_1t_2}-k^2$ in the complex plane, it is straightforward to see the main issues for iterative methods. First of all the size of the triangle grows as $h^{-2}$ which can be expected with the Laplacian operator involved. This bad conditioning destroys the efficiency of Krylov subspace methods. It would not necessarily be an issue for a multigrid method, however another difficulty is the indefiniteness of the matrix. The negative Helmholtz shift $-k^2$ drives the upper branch of the pitchfork closer towards or even past the origin. This makes the coarse grid correction in multigrid highly unstable due to a possible numerical resonance at a coarser level as was reported in \cite{JCP-paper,polynomialsmoother}. A common solution is a preconditioned Krylov subspace method where another matrix $M_h$ is defined such that,
\begin{equation*}
M_h^{-1}H_hu_h=M_h^{-1}f_h,
\end{equation*}
can be easily solved instead. The preconditioning matrix $M_h$ is chosen such that it is efficiently invertible with a fast multigrid method and such that the preconditioned matrix $M_h^{-1}H_h$ is well conditioned, that is, its eigenvalues are clustered around $1$ away from the origin. The complex shifted Laplacian $M^{CSL} = -\triangle -\beta k^2$ has been a successful choice introduced by Erlangga \cite{EVO04} for Sommerfeld radiation conditions. Simply shifting the Laplacian downwards into the complex plane fixes the coarse grid correction in multigrid. In \cite{JCP-paper} this idea was used with ECS boundary conditions, together with the introduction of the closely related complex stretched grid (CSG) operator $M^{CSG}$, that is constructed by discretizing the original Helmholtz operator $-\triangle -k^2$ on a different complex stretched domain,
\begin{equation}\label{eq:domaincsg}
 z(x) = \left\{
  \begin{array}{ll}
    xe^{\imath\thetacsg}, & \hbox{$x \in [0,1]$;} \\
    e^{\imath\thetacsg}+(x-1)e^{\imath\thetacont}, & \hbox{$x \in (1,R]$.}
  \end{array}
\right.
\end{equation}
This domain is complex scaled over $e^{\imath\thetacsg}$ in the interior region $[0,1]$; the exterior complex contour has the same scaling $e^{\imath\thetacont}$, see Figure~\ref{fig:domaincsg}. The spectrum of the discretized operator $M_h^{CSG}=L_h^{CSG}-k^2I_h$ is pitchfork shaped as the original Helmholtz operator $H_h=L_h-k^2I_h$, but with the troublesome upper branch deeper in the complex plane, see Figure~\ref{fig:pitchfork_csg}. Indeed, back scaling the entire preconditioning domain over the inner angle $\thetacsg$ with $e^{-\imath\thetacsg}$ returns a regular ECS domain with a real interior region and an ECS layer with a reduced angle $\thetacont-\thetacsg$. Discretizing the Laplace operator on this latter ECS domain gives the scaled matrix $e^{2\imath\thetacsg}L_h^{CSG}-k^2I_h$ and so $M_h^{CSG}=L_h^{CSG}-k^2I_h$ must have a
pitchfork shaped spectrum too, though somewhat more narrow and rotated away from the real axis over an angle $-2\thetacsg$.  Similar to the complex shift $\beta$ in $M^{CSL}$, the exact choice of the interior scaling angle $\thetacsg$ determines the performance of multigrid on the preconditioner versus the overall convergence rate of the preconditioned Krylov subspace method. In \cite{polynomialsmoother} \br{GMRES is suggested as a smoother substitute in} multigrid which permits small angles $\thetacsg$ for the preconditioner $M^{CSG}$. This improves the Krylov subspace convergence significantly. The goal of this paper is to have a better understanding of the spectrum of the preconditioned operator $M_h^{-1}H_h$, where eventually $M_h=M_h^{CSG}$ will be inverted with a multigrid method.

\begin{figure}[h!]
  \includegraphics[width=0.9\textwidth]{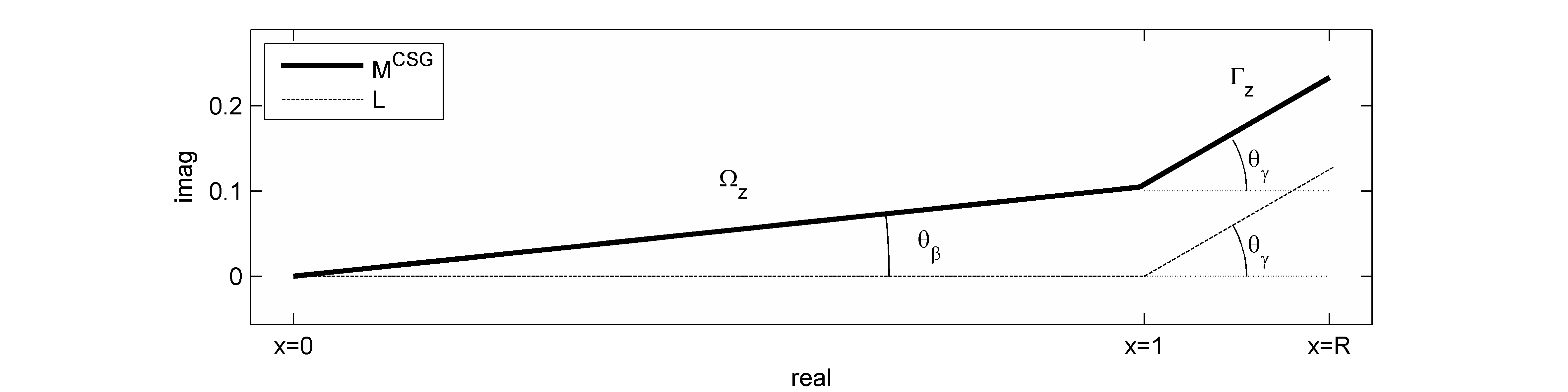}
  \caption{The domain of the CSG preconditioner (solid line)
    \eqref{eq:domaincsg} differs from the original ECS domain (dashed
    line) in the interior region where it is scaled into the complex
    plane by $e^{\imath\theta_\beta}$. The exterior complex contour
    has the same scaling $e^{\imath\theta_\gamma}$ as the original ECS
    domain.}\label{fig:domaincsg}
\end{figure}

\begin{figure}[h!]
\begin{center}
\includegraphics[width=\textwidth]{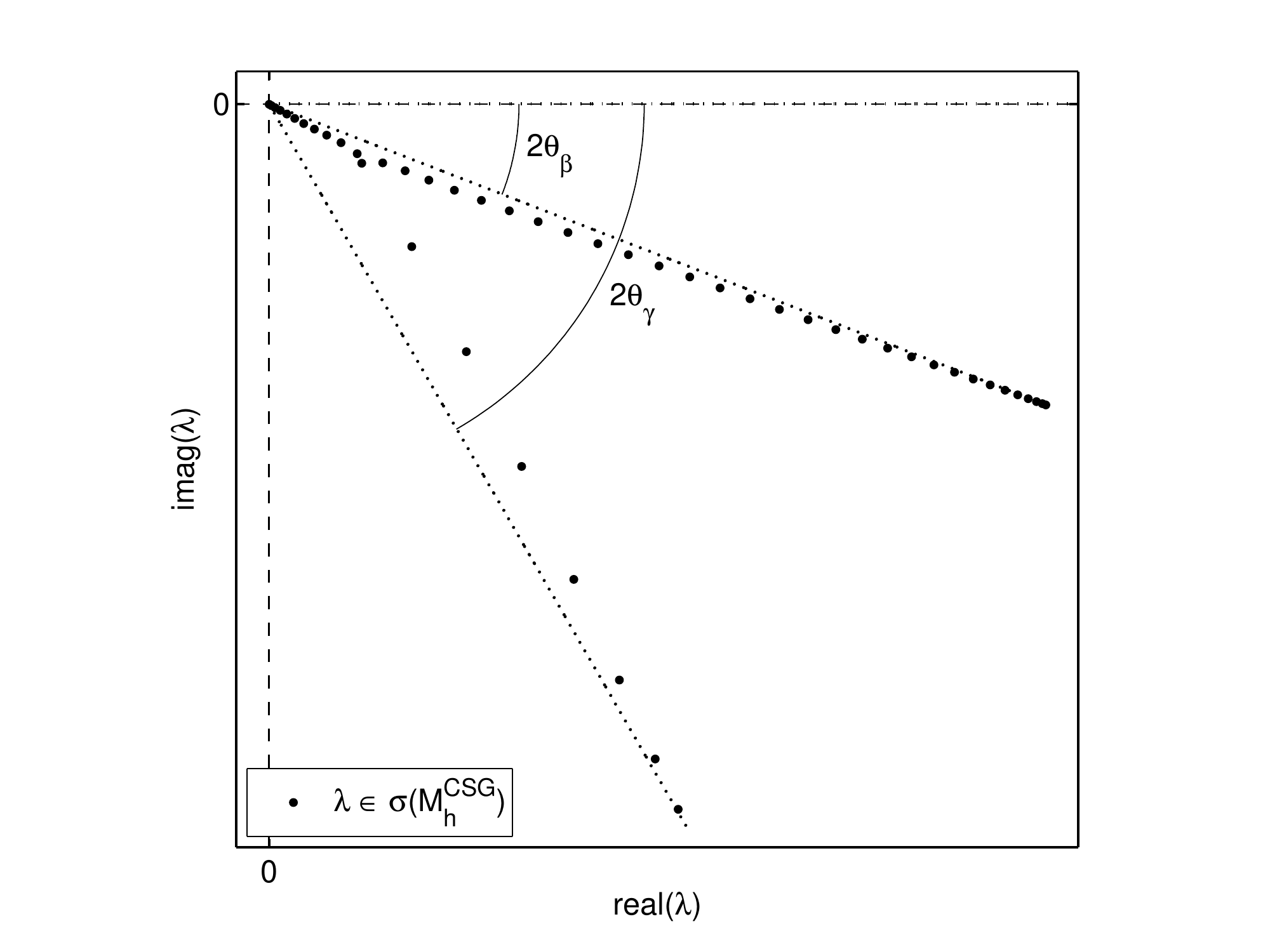}\caption{The spectrum ($\bullet$) of the Laplacian discretized on the preconditioning domain in Figure~\ref{fig:domaincsg} is pitchfork shaped too, but with the upper branch rotated away from the real axis.}
\label{fig:pitchfork_csg}
\end{center}
\end{figure}
%
%
\section{Eigenvalues of the 1D Laplacian on the complex domain}\label{sec:continuouseigenvalues}
In this section we discuss the eigenvalues of the Helmholtz problem formulated on an ECS domain as in Figure~\ref{fig:domain}. To this aim we first consider the Laplacian on a one-dimensional stretched domain
\begin{equation}\label{eq:contour}
z(x) = \int_0^x q(t) dt,
\end{equation}
in the complex plane. \br{In order to simplify the discussion in this section we purposely use the integral representation \eqref{eq:contour} for the ECS domain, as opposed to the formulation in \eqref{eq:domain}.}
We are interested in eigenmodes \br{of the Laplacian},
\begin{equation*}
  -\frac{d^2}{dz^2} u_j = \lambda_j u_j,
\end{equation*}
with Dirichlet boundary conditions $u_j(0)=0$ and $u_j(z(R))=0$.  For
the remainder of this discussion we will drop the subscript $j$ on $u$
and $\lambda$. After applying the chain rule, the equation becomes
\begin{equation*}
  \left[-\frac{1}{q(x)}\frac{d}{dx} \frac{1}{q(x)} \frac{d}{dx} -\lambda\right] u(x)=0,
\end{equation*}
with $u(0)=0$ and $u(R)=0$.  

\br{The domain is described by Equation~\eqref{eq:contour} with 
\begin{equation}\label{eq:ecsdomain}
q(x) = \begin{cases}
   1  \quad & 0 \le x \le  r,\\
   \gamma \equiv e^{\i\thetacont}  \quad&  r< x \le R,
\end{cases}
\end{equation}
where $r=1$ for the model problem in \eqref{eq:helm}.} We then have a second order ODE with constant coefficients on $[0,r]$ and on $(r,R]$
and the solutions can be written as a linear combination of two
fundamental solution. We denote with $u_1(x)$ the solution on the first
interval and $u_2(x)$ the solution on the second interval.  In the point
$r$ the solutions of the both subdomains need to be matched with the
conditions
\begin{equation*}
\begin{cases}
u_1(r) &= u_2(r), \\
\lim_{\epsilon \rightarrow 0}\frac{1}{q(r-\epsilon)}u_1^\prime(r-\epsilon)   &= \lim_{\epsilon \rightarrow 0} \frac{1}{q(r+\epsilon)}u_2^\prime(r+\epsilon),
\end{cases}
\end{equation*}
where the jump condition on the derivative expresses that $u(z(x))$ needs to be continuously differentiable along the transformed domain $z(x)$.
Solving the equation on each subdomain with boundary conditions $u_1(0)=0$ and $u_2(R)=0$ leads to 
\begin{equation*}
  \begin{cases}
    u_1(x)  = A\sin(x\sqrt{\lambda}), & 0\le x \le r; \\
    u_2(x)  = B\sin((x-R)\gamma\sqrt{\lambda}), & r \le x \le R,
\end{cases}
\end{equation*}
where $A$ and $B$ are unknown coefficients.  The solutions $u_1$ and $u_2$ have to fulfill the matching condition in $r$
\begin{equation*}
  u_1(r) = u_2(r) \quad \text{and}\quad  u_1^{\prime}(r) = \frac{1}{\gamma} u_2^\prime(r).
\end{equation*}
After setting $A=1$, which can be done without loss of generality, and eliminating $B$ with the first condition, the second equation becomes
\begin{equation*} 
\cos\left(r\sqrt{\lambda}\right) - \frac{\sin\left(r\sqrt{\lambda}\right)\cos\left(\left(r-R\right)\gamma\sqrt{\lambda}\right)}{\sin\left(\left(r-R\right)\gamma \sqrt{\lambda} \right)} = 0.
\end{equation*}
After some trigoniometry this leads to the condition 
\begin{equation*}
\sin\left(\sqrt{\lambda}((R-r)\gamma + r) \right)=0,
\end{equation*}
thus to the eigenvalues
\begin{equation*}
  \lambda_j = \left(\frac{j \pi}{\left(R-r\right) \gamma + r}\right)^2,
\end{equation*}
for $j\in\mathbb{N}_0$.

Note that $r + (R-r)\gamma=R_z=z(R)$ is the end point of the complex contour on which we have solved the \br{ODE}. In this point we have enforced Dirichlet boundary condition.  \br{Therefore these eigenvalues belong to eigenmodes that are standing waves} on the domain $[0,R_z]$.  The eigenvalues are independent of the details of the complex contour. If we would have taken a more complicated contour\br{, with e.g.\ quadratical scaling instead of linear scaling over a constant angle $\thetacont$ as in \eqref{eq:ecsdomain},} the eigenvalues would be the same.

Note that the discrete problem, discussed in
Section~\ref{sec:modelproblem}, only approximates the first few
eigenmodes along this line. Then, at a certain point along the line, the
spectrum of the discrete operator will bifurcate into two branches as
shown in Figure~\ref{fig:pitchfork}.  This branch point will be discussed in Section~\ref{ssec:branchpoint}.

The spectrum of the Helmholtz problem with constant wave
number $k$ is now 
\begin{equation}\label{eq:eigenvalues}
  \lambda_j(k) =  \left(\frac{j \pi}{R_z}\right)^2 -k^2. 
\end{equation}
with $j\in\mathbb{N}_0$. These are the eigenvalues of the Laplacian shifted over $-k^2$.

%
%

\section{Eigenvalues of the preconditioned problem on the 1D domain}\label{sec:continuouseigenvaluesprecon}
Let us now look at the eigenvalues of the preconditioning operator $M^{CSG}$.  It is defined on a domain described by  
\begin{equation}\label{eq:csgdomain}
p(x) = \begin{cases}
   \beta \equiv e^{\i \thetacsg},   & 0 \le x  \le  r;\\
   \gamma \equiv e^{\i \thetacont}, & r < x  \le R. 
\end{cases}
\end{equation}
Let us denote the end point of this complex domain as $\tilde{R}_z =
\int_0^R p(t)dt$. The eigenvalues of the preconditioning operator are described by $j^2\pi^2/\tilde{R}_z^2$, using the results of the previous section. So both for the original problem as for the preconditioned operator we have that the eigenvalues lie on a straight line in the complex plane.

Let us assume that for each $j$ the eigenvectors for the domains defined by $p$ and $q$ are the same. Then the
eigenvalues $\mu$ of the preconditioned operator $(M^{CSG})^{-1}H$ can be approximated by
  \begin{equation*}
    \mu_j = \frac{ j^2 \pi^2/R_z^2 -k^2}{j^2 \pi^2/\tilde{R}^2_z -k^2}.
  \end{equation*}
This can be rewritten as
\begin{equation}\label{eq:eigMHapprox}
\mu_j  = \frac{\tilde{R}^2_z}{R^2_z}\frac{ j^2 \pi^2/k^2 -R_z^2}{j^2 \pi^2/k^2 -\tilde{R}_z^2},
\end{equation}
\br{which is the evaluation of a linear fractional transformation $LF:\mathbb{C}\to\mathbb{C}:w\mapsto\frac{\tilde{R}^2_z}{R^2_z}\frac{w -R_z^2}{w -\tilde{R}_z^2}$ in the points $j^2 \pi^2/k^2\in\mathbb{R}$ with $j\in\mathbb{N}_0$. It is a known property in complex analysis that $LF$ maps lines to lines or circles in the complex plane. In this case} we find that the eigenvalues $\mu_j$ form a circle in the complex plane with radius
\begin{equation*}
\frac{|\tilde{R}^2_z|}{|R^2_z|}\left|\frac{R_z-\tilde{R}_z}{2 \Im(\tilde{R}_z)}\right|. 
\end{equation*}
This circle does not include the origin.  Note that the radius of the
circle is independent of the wave number $k$.

However, as the next discussion will show, the assumption that the eigenvectors for a given $j$ on the $p$ and $q$ domains are the same is invalid \br{because the eigenvalues of $(M^{CSG})^{-1}H$ are in fact different from $\mu_j$ in \eqref{eq:eigMHapprox}. Indeed, in order to} understand the spectrum we have to solve the eigenvalues of the operator
\begin{align}\label{eq:MHcontinu}
  \left(-\frac{1}{p(x)}\frac{d}{dx}\frac{1}{p(x)}\frac{d}{dx} -k^2 \right)^{-1}  \left(-\frac{1}{q(x)}\frac{d}{dx}\frac{1}{q(x)}\frac{d}{dx}-k^2 \right)u(x) = \mu u(x),
\end{align}
which is a generalized eigenvalue problem
\begin{equation*}
  \left(-\frac{1}{q(x)}\frac{d}{dx}\frac{1}{q(x)}\frac{d}{dx}-k^2 \right)u(x) = \mu  \left(-\frac{1}{p(x)}\frac{d}{dx}\frac{1}{p(x)}\frac{d}{dx} -k^2 \right)u(x)
\end{equation*}
that becomes, after reordering 
\begin{eqnarray*}
  &&\left[ -\frac{1}{q(x)}\frac{d}{dx}\frac{1}{q(x)}\frac{d}{dx} + \mu\frac{1}{p(x)}\frac{d}{dx}\frac{1}{p(x)}\frac{d}{dx} -(1-\mu)k^2 \right]u(x) = 0,\\
 &\Leftrightarrow & \left[ -\frac{1}{\tilde{q}(x)}\frac{d}{dx}\frac{1}{\tilde{q}(x)}\frac{d}{dx} -(1-\mu)k^2 \right]u(x)  = 0,
\end{eqnarray*}
with $1/\tilde{q}(x)=\sqrt{1/q(x)^2-\mu/p(x)^2}$.

\begin{prop}\label{prop:eigMHcontinu}
  For the model problems with domains defined by \br{$q(r)$ in \eqref{eq:ecsdomain} and $p(r)$ in \eqref{eq:csgdomain}},
  the eigenvalues of the operator in \eqref{eq:MHcontinu} are
\begin{equation*}
  \mu_j =  \frac{s_j^2 -1 }{s_j^2/\beta^2- 1}
\end{equation*}
with
\br{
\begin{equation*}
 s_j = \frac{1}{r}\left(\frac{j\pi}{k} - \gamma(R-r)\right) .
\end{equation*}
}
\end{prop}
\begin{proof}
  The piecewise constant $p$ and $q$ again lead to a second order ODE with constant coefficients
  for $u_1(x)$ on the interval $[0,r]$ and for $u_2(x)$ on the interval $[r,R]$.
\begin{equation*}
\begin{cases}
 - (1 - \mu \frac{1}{\beta^2}) \frac{d^2}{dx^2} u_1  -(1-\mu)k^2 u_1 = 0 & \quad  0 \le x \le r,\\
 -(1 - \mu) \frac{1}{\gamma^2} \frac{d^2}{dx^2} u_2  -(1-\mu)k^2 u_2 = 0 &\quad  r \le x \le R,
\end{cases}
\end{equation*}
with $u_1(0)=0$ and $u_2(R)=0$.  The solutions $u_1$ and $u_2$ are 
\begin{equation*}
\begin{cases}
  u_1(x) =  A\sin\left(k\beta\sqrt{\frac{1-\mu}{\beta^2-\mu}} x \right) &\quad 0 \le x \le r,\\
  u_2(x) =  B\sin\left(k\gamma(x-R) \right) &\quad  r \le x \le R,\\
\end{cases}
\end{equation*} 
that need to be matched by the conditions 
\begin{equation*}
\begin{cases}
  u_1(r) &= u_2(r);\\
 \lim_{\epsilon \rightarrow 0}\frac{1}{\tilde{q}(r-\epsilon)}u_1^\prime(r-\epsilon) &=   \lim_{\epsilon \rightarrow 0} \frac{1}{\tilde{q}(r+\epsilon)}u_2^\prime(r+\epsilon).
\end{cases}
\end{equation*}
Without loss of generality we can choose $A=1$.  Requiring continuity of the solution in $r$ leads to
\begin{equation*}
  B =  \frac{\sin\left( k \beta\sqrt{(1-\mu)/(\beta^2-\mu)} r\right)}{\sin\left(k\gamma(r-R)\right)}.
\end{equation*}
Inserting this in the matching condition for the derivatives leads, after some trigoniometry, to  
\begin{equation*}
  \sin\left(k\beta\sqrt{\frac{1 - \mu}{\beta^2 - \mu}} r + k\gamma\left(R-r\right) \right) = 0.
\end{equation*}
The eigenvalues are the solutions of 
\begin{equation*}
k\beta\sqrt{\frac{1 - \mu}{\beta^2 - \mu}} r + k\gamma\left(R-r\right) = j \pi
\end{equation*}
and we find that
\begin{equation*}
  \mu_j = \frac{s_j^2 -1 }{s_j^2/\beta^2- 1},
\end{equation*}
where 
\begin{equation*}
 s_j = \frac{1}{r}\left(\frac{j\pi}{k} - \gamma(R-r)\right)
\end{equation*}
\end{proof}

\begin{cor}\label{cor:parametric}
  The eigenvalues of the preconditioned operator $(M^{CSG})^{-1}H$ lie on a parametric curve $t:[-\Re(\eta), \infty) \rightarrow \mathbb{C}$ that
  maps $t$ to
\begin{equation*}
\frac{(t-\i\Im(\eta))^2-1}{(t-\i\Im(\eta))^2/\beta^2 - 1},
\end{equation*}
with $\eta = \gamma(R/r-1)$.  When $\Im(\gamma)=0$, the curve \br{lies on a circle} through $0$, \br{$\beta^2$} and $1$.
\end{cor}
\begin{proof}
Splitting $s_j$ \br{in Proposition~\ref{prop:eigMHcontinu}} into a real and imaginary part leads to 
\begin{equation*}
s_j = \frac{1}{r}\left(\frac{j\pi}{k} -  \Re\left(\gamma\right)(R-r)\right) - \i \frac{1}{r}\Im\left(\gamma\right)(R-r).
\end{equation*}
Define $\eta = \gamma(R/r-1)$, then for each $j\in\mathbb{N}_0$ there is a $t \in [-\Re(\eta), +\infty)\subset\mathbb{R}$ such that
$s_j = t -\i (R/r-1)\Im\left(\gamma\right)$. \br{In the limit when there is no exterior complex scaling, i.e.\ $\Im(\gamma)=0$, we can reparametrize by $\hat{t}=t^2\in\mathbb{R}$ and so the curve reduces to a linear fractional transformation. It follows that the real line is mapped to a circle through the points $0$, \br{$\beta^2$} and $1$.}
\end{proof}

It is important to note that changing the wave number $k$ does not
alter the parametric curve.  Indeed, changing $k$ only modifies the
real part of $s_j$ which leads to a different particular choice
$t$ that gives the position of the eigenvalue $\mu_j$ on the curve.  This means that there is an upper bound for the condition number $\kappa=\frac{\max_j|\mu_j|}{\min_j|\mu_j|}$ of the
preconditioned problem that is independent of $k$ \br{and suggests a fast convergence of the preconditioned Krylov subspace method}.

The spread of the eigenvalues on the parametric curve can change as a function of $k$.  First note that in the limit $j\rightarrow \infty$ the \br{eigenvalues $\mu_j$} go to $\beta^2$. In a similar way, the eigenvalues will accumulate near $\beta^2$ as $k \rightarrow 0$. In the other case, as $k$ \br{gets larger}, the smoothest eigenvalues $\mu_j$ with $j\ll k$ will go to the other end of the curve,
\begin{equation*}
\lim_{j/k\to0}\mu_j = \frac{\gamma^2(R/r-1)^2-1}{\gamma^2(R/r-1)^2/\beta^2-1}\approx 1,
\end{equation*}
\br{leading to a spectrum that is completely spread over the curve.}

\br{We clearly observe this behavior in Figure~\ref{fig:eigMA} where the spectrum of the preconditioned system is plotted for different values of $k$ for the one-dimensional model problem \eqref{eq:helm} with $r=1$, $R=1.25$, outer ECS angle $\thetacont=\frac{\pi}{6}$, inner angle for the preconditioner $\thetacsg=0.18\approx\frac{\pi}{17}$. The circles mark the first $80$ eigenvalues of the continuous problem, as given by Proposition~\ref{prop:eigMHcontinu}. They lie on the parametric curve derived in Corollary~\ref{cor:parametric} which is visualized with a solid line. For a small wave number $k=0.4$, in the upper left subfigure, all eigenvalues $\mu_j$, with $1\leq j\leq80$ lie close to $\beta^2$. In the upper right plot with $k=6.4$ we see that $\mu_{1}$ and $\mu_{2}$ almost reach the other end of the curve, while the remaining eigenvalues $\mu_j$ with $j\geq3$ still lie closer to $\beta^2$. Finally in the lower two subfigures the eigenvalues are more spread along the curve as $k$ grows larger.}

In a similar way the 2D eigenvalues, or for any higher dimension, will be bounded by a parametric curve that is independent of the wave number. In contrast to the one-dimensional case the 2D spectrum will fill up the region bounded by the curve with eigenvalues.

\begin{figure}[h!]
\begin{center}
\includegraphics[width = \textwidth]{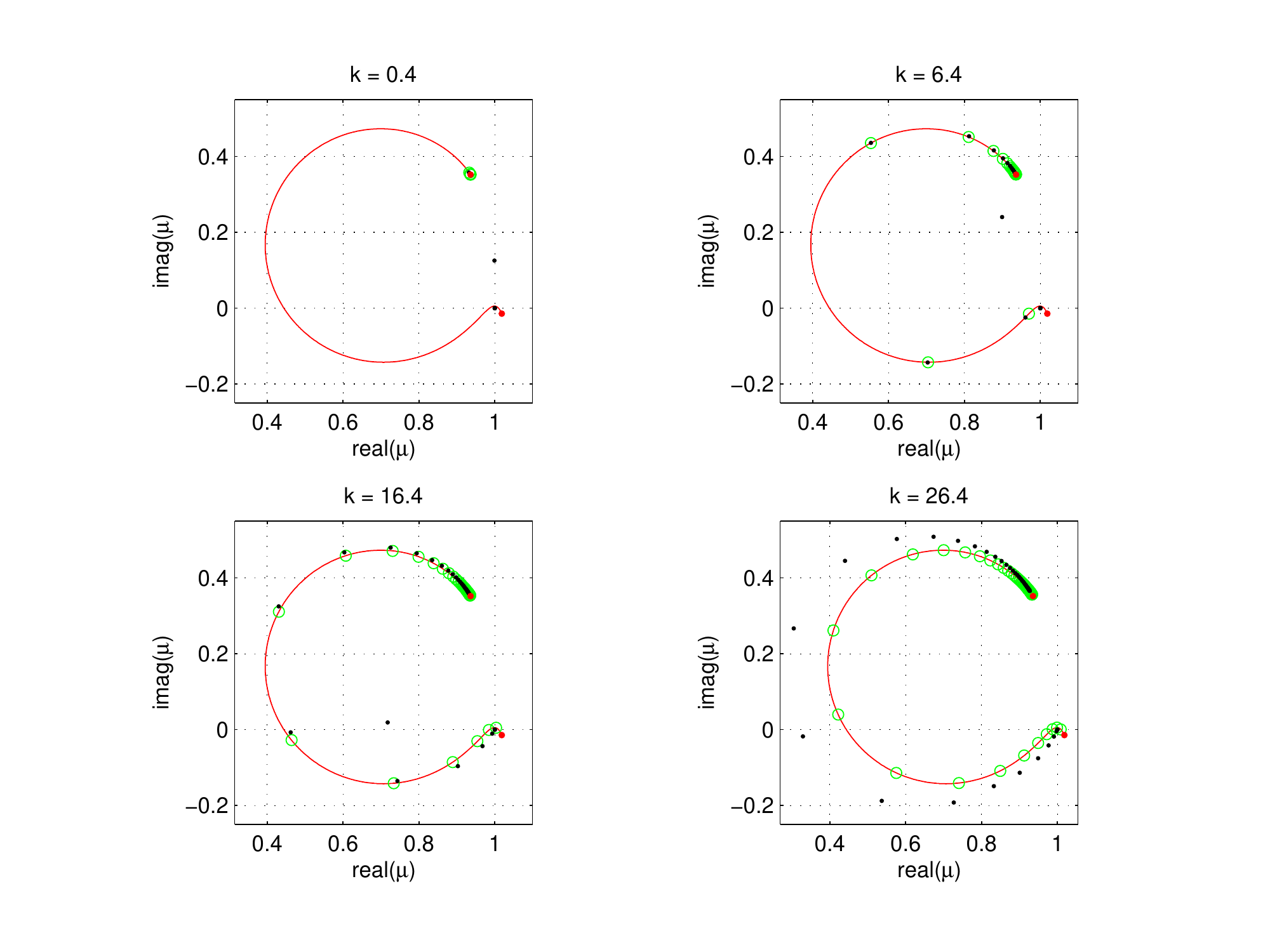}
\caption{First $80$ Eigenvalues of the one-dimensional preconditioned system for different values of the wave number $k$. For the continuous operator ($\circ$) the spectrum lies on a parametric curve (red line) in the complex plane. When $k$ increases the eigenvalues spread over the curve. The spectrum of the discrete operator ($\bullet$) approximates the continuous case for $k=0.4,6.4$ and $16.4$ leading to an initally bounded condition number in Figure~\ref{fig:conditionnumber}. For $k=26.4$ the discrete eigenvalues lie outside the curve resulting in a larger condition number. For higher dimensional problems the region bounded by the curve will fill up with eigenvalues.}
\label{fig:eigMA}
\end{center}
\end{figure}

%
%
\section{Discrete operator}\label{sec:discrete}
\subsection{Deviations from the continuous problem}\label{ssec:deviation}
However, when the problem is discretized, with for example finite
differences, the Krylov convergence rate can differ significantly from
the bounds predicted by the analysis of the continuous problem in
Section~\ref{sec:continuouseigenvalues} and \ref{sec:continuouseigenvaluesprecon}.

\br{
Indeed, \br{Figure~\ref{fig:gmres_iterations} shows} the number of GMRES iterations to solve the Helmholtz problem with discretization matrix $H_h$, preconditioned with the complex shifted grid matrix $M_h^{CSG}$ that is exactly inverted, as a function of the wave
number $k$. The results are for a 1D (solid line) and a 2D (dashed line) problem with $80$ grid points per dimension, $n=64$ interior grid points and $m=16$ additional points for the ECS layer. For the 2D problem we recognize three regions in the
convergence rate.  First, between $k=0$ and $k\approx16.4$, the number of iterations is ramped up from a few to about $18$. The second region is between $k\approx16.4$ and $k\approx21$, where the number of iterations remains constant. Finally from $k=21$ on the number of iteration rises again.  This is in contrast to the analysis of the continuous problem that predicts a convergence rate that is independent of $k$, once it is large enough. These observations can be explained with the help of the spectrum of the discrete preconditioned system $M_h^{-1}H_h$ of the one-dimensional problem shown in Figure~\ref{fig:eigMA}.
}

\br{
For the wave numbers in the first region, $0\leq k\leq16.4$, the spectrum of the discrete problem ($\bullet$) is a good approximation for the spectrum of the continuous continuous problem ($\circ$), except for one or two spurious eigenvalues. As a consequence, the spectrum lies along the parametric curve (solid line). The initial growth in the convergence behavior is due to the fact that the eigenvalues are accumulated near $\beta^2$ for small $k$ and start spreading over the curve towards $1$ for larger $k$, leading to an increasing condition number $\kappa = \frac{\max_j{|\mu_j|}}{\min_j{|\mu_j|}}$ where $1\leq j\leq80$. In Figure~\ref{fig:conditionnumber} the condition number is plotted as a function of the wave number $k$. Indeed, first the condition number is close to $1$ until eigenvalue $\mu_{80}$ moves significantly along the curve reaching the point with the smallest possible absolute value for $k\approx2.5$. The next eigenvalue $\mu_{79}$ gives rise to a second peak around $k\approx5$ when it reaches that same minimal point on the curve. As the spectrum starts spreading more equally over the curve, see e.g.\ lower left subfigure in Figure~\ref{fig:eigMA}, the condition number seems to converge to $\kappa\approx 2.5$ and the number of GMRES iterations stagnates around $18$ for the second region $k\approx16.4$ to $k\approx21$ on Figure~\ref{fig:gmres_iterations}.
}

\br{
However, from $k\approx16.4$ on the discrete condition number starts behaving differently from what the continuous eigenvalues predict. Around $k\approx21$ it even grows above the upper bound given by the curve. This is because the eigenvalues of the discrete preconditioned system start deviating from the continuous spectrum, see e.g.\ lower right subfigure in Figure~\ref{fig:eigMA}. Whereas the continuous eigenvalues remain on the curve, the discrete spectrum grows outside the curve from a certain critical wave number. This divergence between the continuous and the discrete preconditioned spectrum finds its origin in the spectrum of the original Helmholtz operator. At the end of Section~\ref{sec:continuouseigenvalues} we mentioned that only the first few eigenvalues $\lambda_j$ in \eqref{eq:eigenvalues} of the continuous Helmholtz operator $H$ are well approximated by the eigenvalues of the discretization matrix $H_h$ given by \eqref{eq:eigcond}. In Figure~\ref{fig:pitchfork} we see that the continuous spectrum ($\times$) lies along a line in the complex plane, whereas the discrete spectrum ($\bullet$) branches from this line at a certain point $t_b\in\mathbb{C}$. The location of this point will be an indicator for the start of the third region in the convergence behavior of GMRES.
}
\br{
For the 2D problem the three different convergence regions are more pronounced than for the 1D problem because the region bounded by the curve gets filled with extra eigenvalues.
}

\begin{figure}[h!]
\begin{center}
\includegraphics[width = \textwidth]{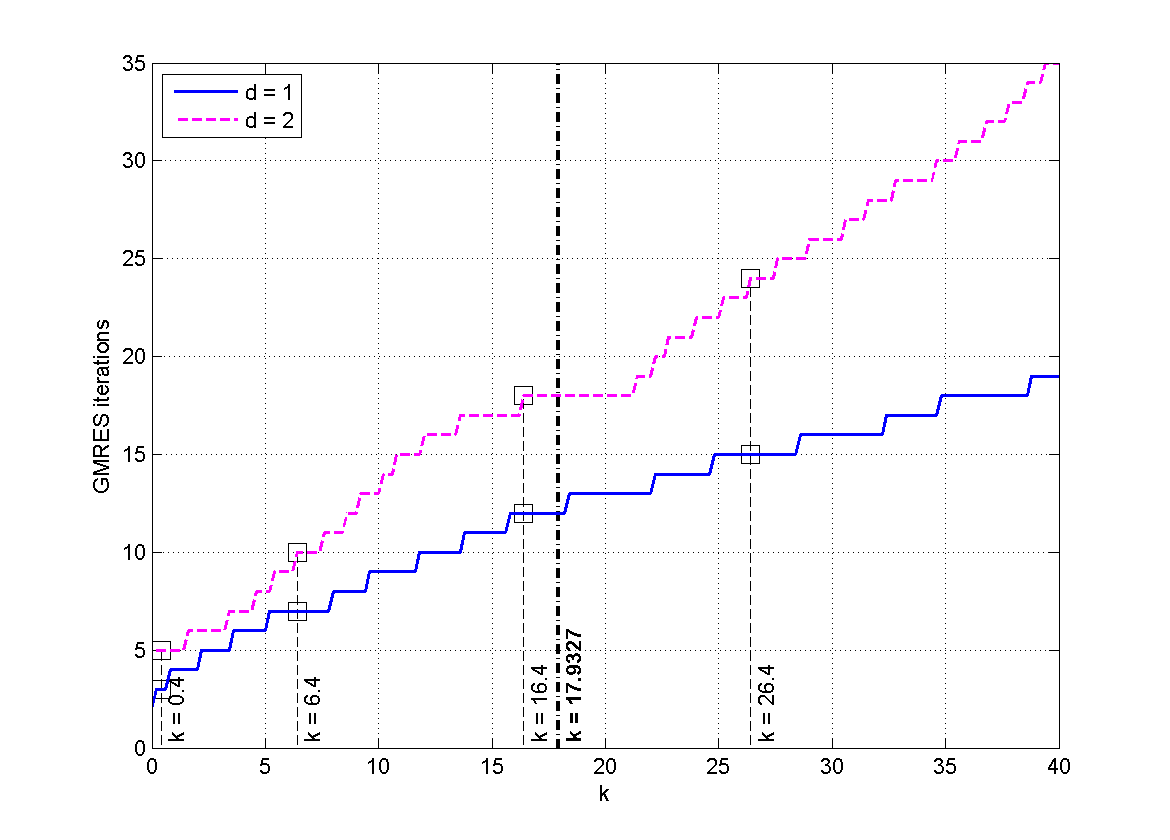}
\caption{Number of GMRES iterations to solve the one-dimensional (solid line) and two-dimensional (dashed line) Helmholtz problem as a function of the constant wave number $k$ with $80$ grid points per dimension. For the 2D problem three regions are clearly distinguishable. Initally the number of iterations increases with $k$ until a stagnation of 18 iterations is reached around $k=16.4$. For wave numbers larger than $k\approx21$ the number of iterations starts increasing again. Four values of $k$ for which the spectrum of the preconditioned system is shown in Figure~\ref{fig:eigMA} are marked ($\square$).}
\label{fig:gmres_iterations}
\end{center}
\end{figure}

\begin{figure}[h!]
\begin{center}
\includegraphics[width = 0.9\textwidth]{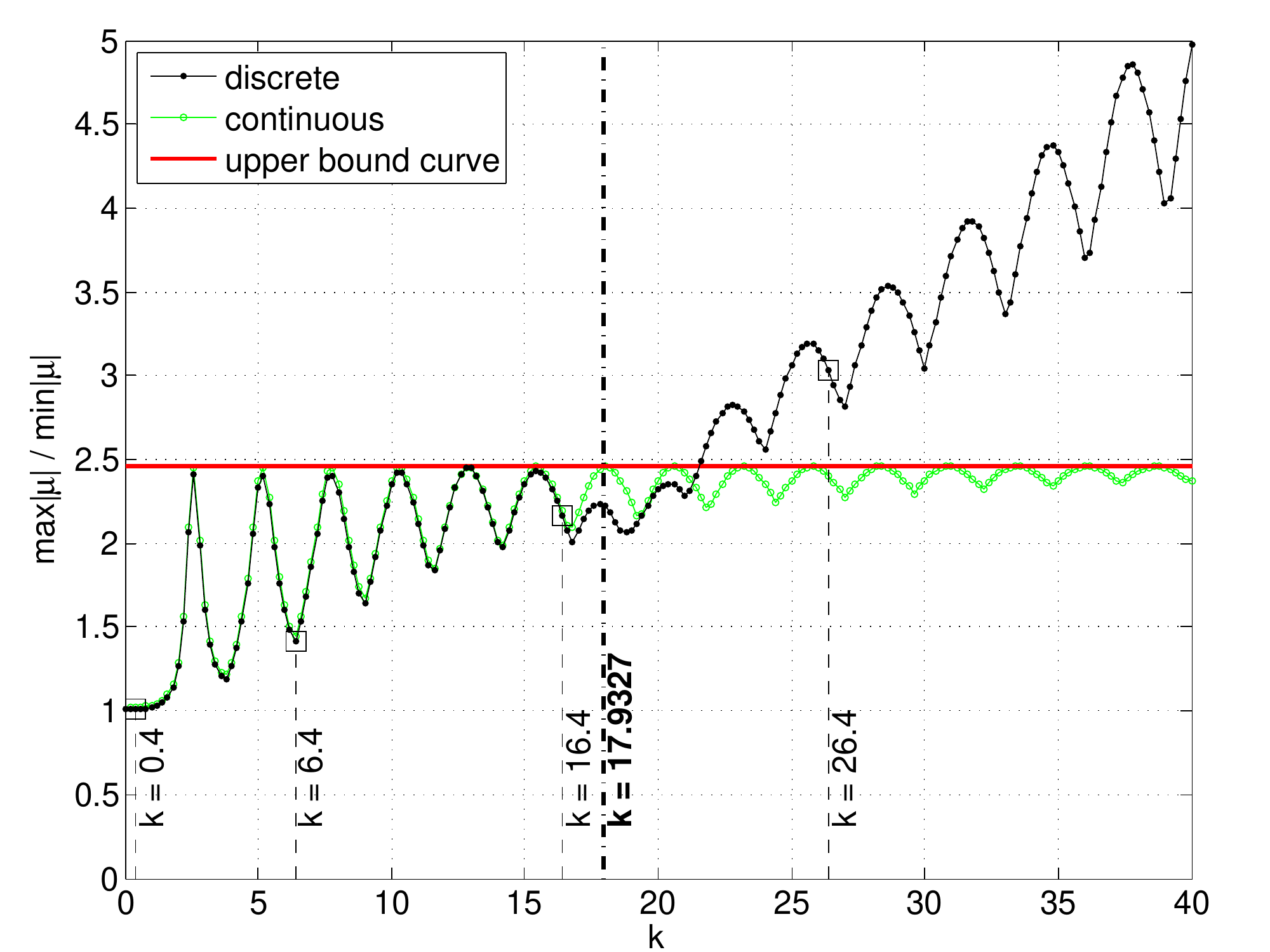}
\caption{Condition number of the continuous ($\circ$) and discrete ($\bullet$) preconditioned system. The condition number of the continuous operator converges to a fixed value determined by the parametric curve in Corollary~\ref{cor:parametric} whereas the discrete operator shows diverging behavior from a certain wave number $k\approx16.4$. Four values of $k$ for which the spectrum of the preconditioned system is shown in Figure~\ref{fig:eigMA} are marked ($\square$).}
\label{fig:conditionnumber}
\end{center}
\end{figure}

%
%
\subsection{Predicting the branch point}\label{ssec:branchpoint}
Next we derive an explicit formula for the \br{approximate} position of the branch point in the spectrum of the discretization matrix, see Figure~\ref{fig:pitchfork}. From this point on, the eigenvalues of the discrete
problem start to deviate significantly from the eigenvalues of the
continuous problem.  This branch point will predict from which $k$ on
we can expect a rising cost of the numerical solution method.  A
surprising result of this section is that this branch point does not
shift with the order of the discretization.

\begin{prop}\label{prop:branchpoint}
  Consider a one-dimensional grid as in \eqref{eq:ecsgrid} defined on an ECS domain consisting of two parts, $[0,r]$
  for the interior with $n$ grid points, and the complex interval $[r,R_z]=[r,z(R)]$ for the
  exterior part with $m$ grid points. The smallest eigenvalues of the negative
  Laplacian discretized on this ECS grid $-L_h$, with zero Dirichlet conditions at the
  boundaries $0$ and $R_z$, lie along the complex line
\begin{align*}
t(\rho) = \left(\frac{\rho}{R_z}\right)^2,  \quad \mbox{with }\rho>0,
\end{align*}
close to the eigenvalues of the continuous operator $t_j = \left(\frac{j\pi}{R_z}\right)^2$ with $j\in\mathbb{N}_0$.
For larger eigenvalues the spectrum of $-L_h$ splits into two branches around
the point $t_b = \left(\frac{\rho_b}{R_z}\right)^2$ with
\begin{align*}
\rho_b = \frac{|R_z|^2}{r\Im(R_z)}W\left(4n\Im(R_z)\left|\frac{1}{R_z}\sqrt{\frac{R-r}{R_z-R}}\right|\right)
\end{align*}
where $R_z=z(R)$ and $W(.)$ is the Lambert-W function.
\end{prop}
\begin{proof}
  The spectrum of the Laplacian discretized on an ECS grid has a
  pitchfork shape. In order to find the point where the pitchfork
  splits we start from the condition \eqref{eq:eigcond} that has the eigenvalues of the discrete Laplacian $-L_h$ as solutions. In the rest of this discussion we will consider the scaled Laplacian $L = -h^2L_h$. The eigenvalues of $L$ are again
the solutions of \eqref{eq:eigcond}, but with
$p(t)=\frac{1}{2}\arccos(1-\frac{t}{2})$,
$q(t)=\frac{1}{2}\arccos(1-\frac{t}{2}\gamma^2)$ instead. The
condition \eqref{eq:eigcond} is equivalent to
\begin{align}\label{eq:eigcond1}
F_1(t) \equiv \sin(2n p(t))\cos(2m q(t))\cos(q(t))+\cos(p(t))\cos(2n p(t))\sin(2m q(t)) = 0,
\end{align}

Since we are interested in the smallest eigenvalues of $L$ we can use the approximate condition
\begin{align*}
F_2(t) \equiv \sin\left(\left(n+m\gamma\right)\sqrt{t}\right)-\frac{1}{2}\sin\left(n\sqrt{t}\right)\cos\left(m \gamma\sqrt{t}\right)\tan\left(\frac{\sqrt{t}}{2}\right)\sqrt{t}\varepsilon = 0,
\end{align*}
with $\varepsilon=\gamma-1$.\\

This is easily derived from \eqref{eq:eigcond1} by using the Taylor series $p(t) = \frac{\sqrt{t}}{2} +\mathcal{O}(|t|^{3/2})$ and $q(t) = \gamma\frac{\sqrt{t}}{2} +\mathcal{O}(|t|^{3/2})$,
for $|t| \ll 1$ and substituting $\gamma \equiv 1+\varepsilon$. Moreover, since $|\varepsilon|<1$ for realistic exterior complex scaling with an ECS angle $\theta_\gamma<\frac{\pi}{4}$, we have used $\cos\left(\gamma\frac{\sqrt{t}}{2}\right) = \cos\left(\frac{\sqrt{t}}{2}\right) -\frac{1}{2}\sin\left(\frac{\sqrt{t}}{2}\right)\sqrt{t}\varepsilon +\mathcal{O}(|t\varepsilon^2|)$.\\

We will now look at the evaluation of function $F_2$ along the complex line
\begin{equation*}
t(\rho) = \left(\frac{\rho}{n + m\gamma}\right)^2 \quad \mbox{with }\rho>0. 
\end{equation*}
It returns real numbers between $-1$ and $1$ for the first term of $F_2$, and complex numbers for the second term. The
latter term is small, for small $\rho$, and thus the roots of $F_2$ will approximately be the roots $t_j=\left(\frac{j\pi}{n + m\gamma}\right)^2$, with $j\in\mathbb{N}_0$, of the first term. The eigenvalues will branch from the line $t(\rho)$ when the second term of $F_2$ becomes more important. This is when,
\begin{align*}
& |\frac{1}{2}\sin\left(n\frac{\rho}{n+m\gamma}\right)\cos\left(m\gamma\frac{\rho}{n+m\gamma}\right)\tan\left(\frac{\rho}{2(n + m\gamma)}\right)\frac{\rho}{n + m\gamma}\varepsilon| \approx 1, \\
\Leftrightarrow& |\frac{\varepsilon}{8}\sin\left(\rho\left(1-\frac{2m\gamma}{n+m\gamma}\right)\right)\left(\frac{\rho}{n + m\gamma}\right)^2| \approx 1,
\end{align*}
and after using the identity $R_z=r+(R-r)\gamma$,
\begin{align*}
\Leftrightarrow& |\frac{\varepsilon}{8}\sin\left(\rho\left(\frac{2(R_z-r)}{R_z}-1\right)\right)\left(\frac{\rho h}{R_z}\right)^2| \approx 1, \\
\Leftrightarrow& \frac{h|\sqrt{\varepsilon}|}{4r|R_z\Im(\frac{1}{R_z})|} \rho r|\Im(\frac{1}{R_z})| e^{\rho r|\Im(\frac{1}{R_z})|} \approx 1, \\
\Leftrightarrow& \rho \approx \frac{W(c)}{r|\Im(\frac{1}{R_z})|},
\end{align*}
where $W(c)$ is the Lambert-W function, \br{defined such that $c=W(c)e^{W(c)}$,} and evaluated in
\begin{equation*}
c=\frac{4r|R_z\Im(\frac{1}{R_z})|}{h|\sqrt{\varepsilon}|}= 4n\Im(R_z)\left|\frac{1}{R_z}\sqrt{\frac{R-r}{R_z-R}}\right|,
\end{equation*}
with $\varepsilon = \gamma-1=\frac{R_z-R}{R-r}$. The point $t_b$ along the line $t(\rho)$ where the pitchfork splits into two branches is now approximately given by
\begin{align*}
& \rho_b = \frac{W(c)}{r|\Im(\frac{1}{R_z})|} = \frac{|R_z|^2}{r\Im(R_z)}W\left(4n\Im(R_z)\left|\frac{1}{R_z}\sqrt{\frac{R-r}{R_z-R}}\right|\right),\\
\Rightarrow & t_b = \left(\frac{\rho_b}{n+m\gamma}\right)^2 = \left(\frac{\rho_b h}{R_z}\right)^2.
\end{align*}
So for the eigenvalues of the unscaled operator $L$ we have $t_b  = \left(\frac{\rho_b}{R_z}\right)^2$.
\end{proof}

The point $t_b=\left(\frac{\rho_b}{R_z}\right)^2$ predicts the point
in the spectrum of the discrete Laplacian where the pitchfork
splits. The smallest eigenvalues lie close to
$\frac{j^2\pi^2}{R_z^2}$, with $j\in\mathbb{N}_0$ such that
$j\pi\leq\rho_b$. This is illustrated in Figure~\ref{fig:splitpoint}
where the 32 smallest eigenvalues are plotted for three different grid
sizes $n$, together with the branch point predictions $t_b$, for the
ECS domain $[0,r]\cup(r,R_z]=[0,1]\cup(1,1+0.25e^{\imath\pi/6}]$.
\begin{figure}[h!]
\begin{center}
\includegraphics[width=\textwidth]{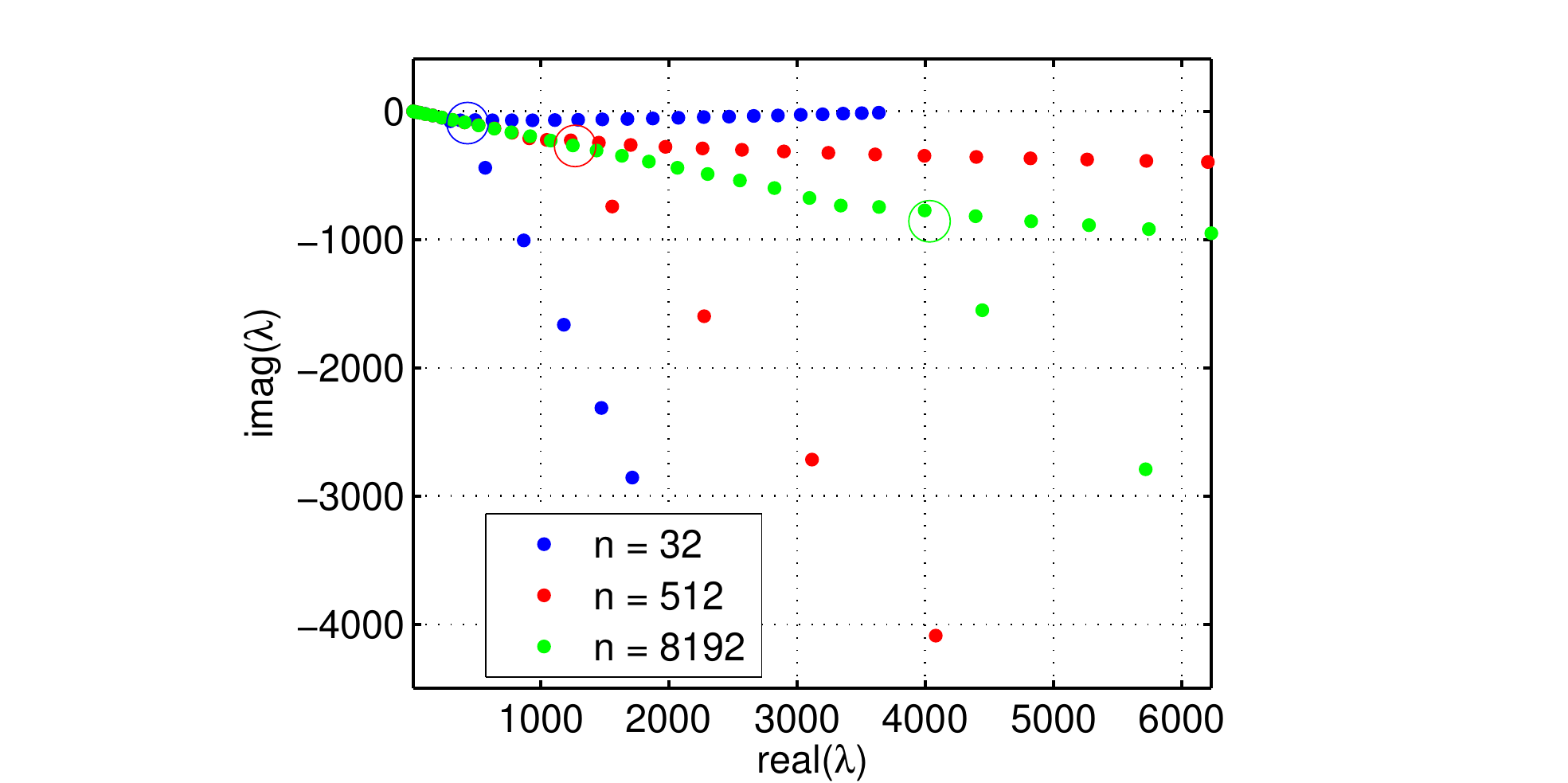}
\caption{The 32 smallest eigenvalues ($\bullet$) of the discretized
  Laplacian for $n=32$, $n=512$ and $n=8192$. The branch point $t_b$ ({\Large $\circ$}) in
  the spectrum moves further in the complex plane as predicted by the
  formula in Proposition~\ref{prop:branchpoint}.}\label{fig:splitpoint}
\end{center}
\end{figure}

Figure~\ref{fig:splitpoint_abs} shows the distance to the origin $|t_b|$ of the predicted branch point
as a function of the interior grid size $n$, the other domain
parameters are fixed. The branch point was also detected
experimentally by measuring the deviation from the line $t(\rho) =
\left(\frac{\rho}{R_z}\right)^2$ with $\rho>0$. As the grid size $n$
increases, the tail of the pitchfork grows proportional to the square
of the Lambert W-function, $|t_b| \sim W(n)^2$, and not according to
the order of discretization.

\begin{figure}[h!]
\begin{center}
\includegraphics[width=\textwidth]{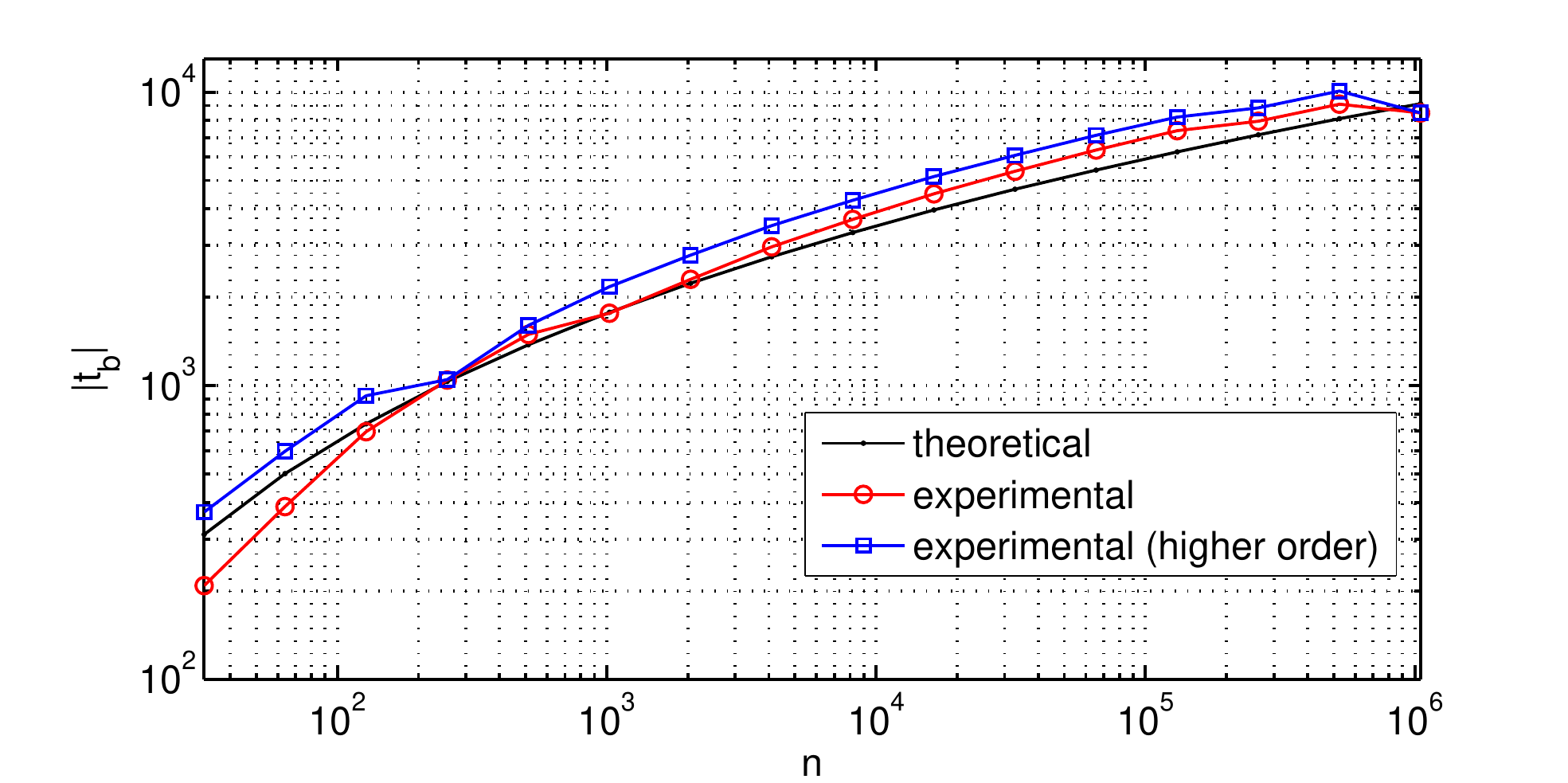}
\caption{The absolute value of the branch point $t_b$ of the pitchfork in the spectrum of the discretized Laplacian for $n=2^{j}$ with $j=5,\ldots,20$. As predicted by the formula ($\bullet$) in Proposition~\ref{prop:branchpoint}, measured experimentally ($\circ$) and with a higher order scheme in the turning point $r$ ($\square$). As the grid size $n$ increases, the length of the typical line of smooth eigenvalues grows proportional to the square of the Lambert W-function, $|t_b| \sim W(n)^2$, and not as the order of the discretization.}\label{fig:splitpoint_abs}
\end{center}
\end{figure}


\br{
The spectrum of the indefinite Helmholtz discretization matrix $H_h$ is achieved by shifting the pitchfork spectrum of the Laplacian to the left in the complex plane over a distance determined by the wave number $k$. We can now predict the value $k$ for which the eigenvalues of the discrete preconditioned system $M_h^{-1}H_h$ start growing outside the parametric curve of the continuous preconditioned spectrum. Since $t_b$ marks the point in the pitchfork shaped spectrum of $H_h$ where the badly-approximated continuous eigenvalues lie on the right hand side, we expect this to happen when $t_b$ is of the same order of magnitude as the smooth eigenvalues, this means when $|t_b|=k^2$. In Figures~\ref{fig:gmres_iterations} and \ref{fig:conditionnumber} the critical wave number $k_b\equiv\sqrt{|t_b|}=17.9327$ is marked. Indeed, for wave numbers $k\leq k_b$ the eigenvalues of the continuous preconditioned operator lead to a good prediction of the effective discrete condition number and the rise in GMRES iterations is a direct consequence of the spreading along the parametric curve. After a short region of stagnation where the curve is densely filled with eigenvalues, the number of iterations increases again, however, this time due to the growing deviaton of the discrete eigenvalues from the continuous eigenvalues. In the next section we will study the convergence behavior for wave numbers in this third region with numerical experiments.
}

%
%
%
\section{Numerical experiments}\label{sec:numerical}

In this section we analyze the performance of a Krylov subspace method
applied to a Helmholtz problem with constant wave number $k$ in a square domain
$\Omega=(0,1)^2$ with a centered point source
\begin{equation}\label{eq:helm2D}
f(x,y) =
\begin{cases}
1,\quad \mbox{if } x=1/2=y,\\
0,\quad \mbox{elsewhere}.
\end{cases}
\end{equation}
This two-dimensional problem is built with Kronecker products of the
one-dimensional model problem \eqref{eq:helm} with outgoing wave
boundary conditions in every direction. All boundaries are therefore
extended with an ECS layer with an angle $\thetacont = \frac{\pi}{6}$
to absorb outgoing waves. The results from the previous sections are
still useful if we take into account that the spectrum of the
two-dimensional operator is the set of all possible sums of two
eigenvalues of the one-dimensional case. This means the spectrum of
the discretization matrix $H_h$ looks like a sum of pitchforks now, as
discussed in Section~\ref{ssec:deviation}, with three points close to
the real axis $t=-k^2$, $t=\frac{4}{h^2}-k^2$ and
$t=\frac{8}{h^2}-k^2$.  These eigenvalues correspond respectively to
the smoothest mode on the domain, the eigenmode which is oscillatory in
one dimension and smooth in the other and finally a mode that is oscillatory in
both dimensions.

The complex stretched grid preconditioning matrix $M_h^{CSG}$ is
constructed by discretizing the problem on a complex stretched grid
with a small inner angle $\thetacsg = 0.18 \approx \frac{\pi}{17}$. This ensures
that for every level in the multigrid hierarchy the eigenvalues are
bounded away from zero.  The angle for the outer ECS layers is kept at $\thetacont =
\frac{\pi}{6}$ as in the original Helmholtz problem as illustrated in
Figure~\ref{fig:domaincsg}. For a detailed description of the spectrum of
this preconditioning matrix we refer to \cite{JCP-paper} and
\cite{polynomialsmoother}.

As a consequence the preconditioning matrix $M_h^{CSG}$ can
efficiently be inverted with a multigrid method with either \br{three steps of GMRES as a smoother substitute, denoted as GMRES($3$), or a specific
polynomial smoother as suggested in \cite{polynomialsmoother}}. However, because the smoother can differ each
application the actual preconditioner is not the same in every outer
Krylov step. Therefore FGMRES, the flexible GMRES method
\cite{SaadFGMRES}, is used as outer Krylov subspace methods. We
discuss the performance of preconditioned FGMRES before convergence to
a residual norm of order $10^{-6}$. The preconditioning matrix
$M_h^{CSG}$ is approximately inverted with one V(1,1)-cycle with
GMRES($3$) as smoother. The experiments are all run in \textsc{Matlab}\texttrademark\ on two
quad core Intel\texttrademark\ Xeon CPUs (E5462 @ 2.80GHz).

Figure~\ref{fig:krylov_fgmres} shows the convergence results for
preconditioned FGMRES to solve the 2D Helmholtz problem \eqref{eq:helm2D} with a
residual norm below $10^{-6}$ for wave numbers ranging from $k=15$
to $k=180$. For these wave numbers the smoothest eigenvalues of the
discretization matrices have a negative real part. Each curve shows
the same experiment for a fixed grid size $n$ in one dimension, the
grid size of the ECS layer is related as $m=n/4$. Just like the wave numbers the grid sizes are purposely chosen over a wide range as well,
from $n=16$ to $n=2048$ in one dimension, in order to expose the full effect of the preconditioning on the convergence behavior of FGMRES. A discussion on the physical accuracy of the grid sizes lies not within the scope of this
analysis. As we are merely interested in the convergence behavior of
Krylov subspace methods we explain these curves as a function of the
increasing wave number $k$.

For each grid size both the convergence rate and the number of
iterations grow initially as a function of $k$ up to a peak where
$k\approx\frac{2}{h}=2n$. This corresponds to the first critical point
$t=\frac{4}{h^2}-k^2$ where the pitchfork in the spectrum of $H_h$
nearly touches the real axis. These eigenvalues correspond to
eigenmodes that oscillate rapidly in one direction while they are
smooth in the other.  There is now an eigenvalue of $H_h$ near the
origin and the preconditioned system $M_h^{-1}H_h$ obviously suffers
from this too. As the wave number $k$ increases more, the pitchfork is
shifted further to the left in the complex plane. The convergence
improves slightly until the second critical point
$t=\frac{8}{h^2}-k^2$ comes too close to the origin for
$k\approx\frac{2\sqrt{2}}{h}=2\sqrt{2}n$. After this, the spectrum has
completely shifted into the negative real part of the complex plane,
making the Helmholtz matrix negative definite. This is observed on the
curves as a sudden improvement in convergence. As a reference these
experiments were repeated for the smallest grid sizes with regular
GMRES and an exact inversion of the preconditioner $M_h^{CSG}$, in
order to eliminate the effect of the approximate multigrid inverse. In
Figure~\ref{fig:krylov_gmres_exact} the described convergence behavior
is then more pronounced.

\begin{figure}[h!]
\begin{center}
\subfloat[The convergence rate of FGMRES, averaged over all iterations.]{
\includegraphics[width=10cm]{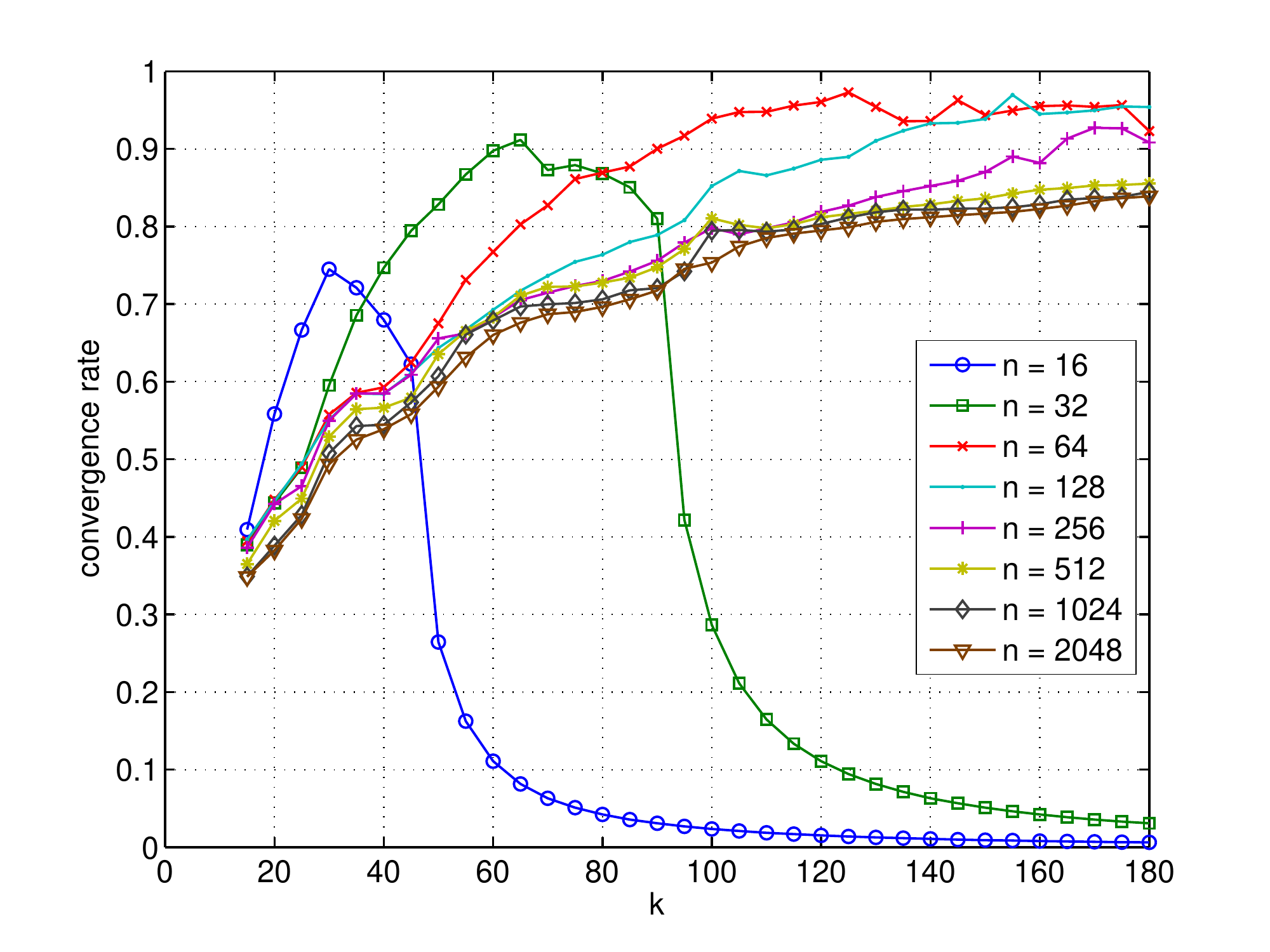}\label{fig:krylov_rate_fgmres}
}\\
\subfloat[The number of FGMRES iterations.]{
\includegraphics[width=10cm]{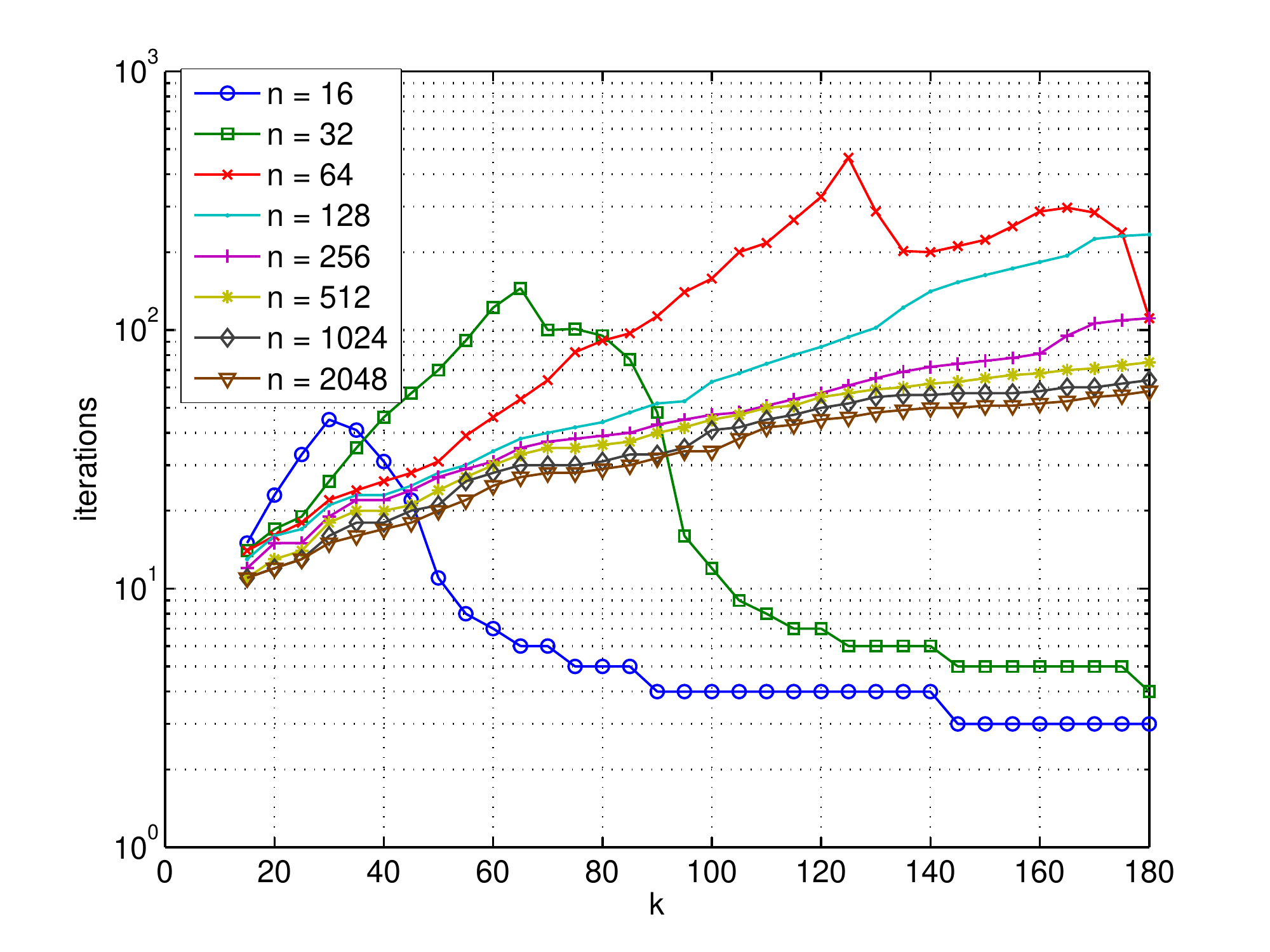}\label{fig:krylov_iter_fgmres}
}
\caption{Convergence results of FGMRES as a function of the wave number $k$ for the homogeneous 2D Helmholtz problem. The preconditioner is approximately inverted with one V(1,1)-cycle. The lines represent different interior grid sizes $16\times16$, $32\times32$,\ldots, $2048\times2048$.}
\label{fig:krylov_fgmres}
\end{center}
\end{figure}

\begin{figure}[h!]
\begin{center}
\subfloat[The convergence rate of GMRES, averaged over all iterations.]{
\includegraphics[width=10cm]{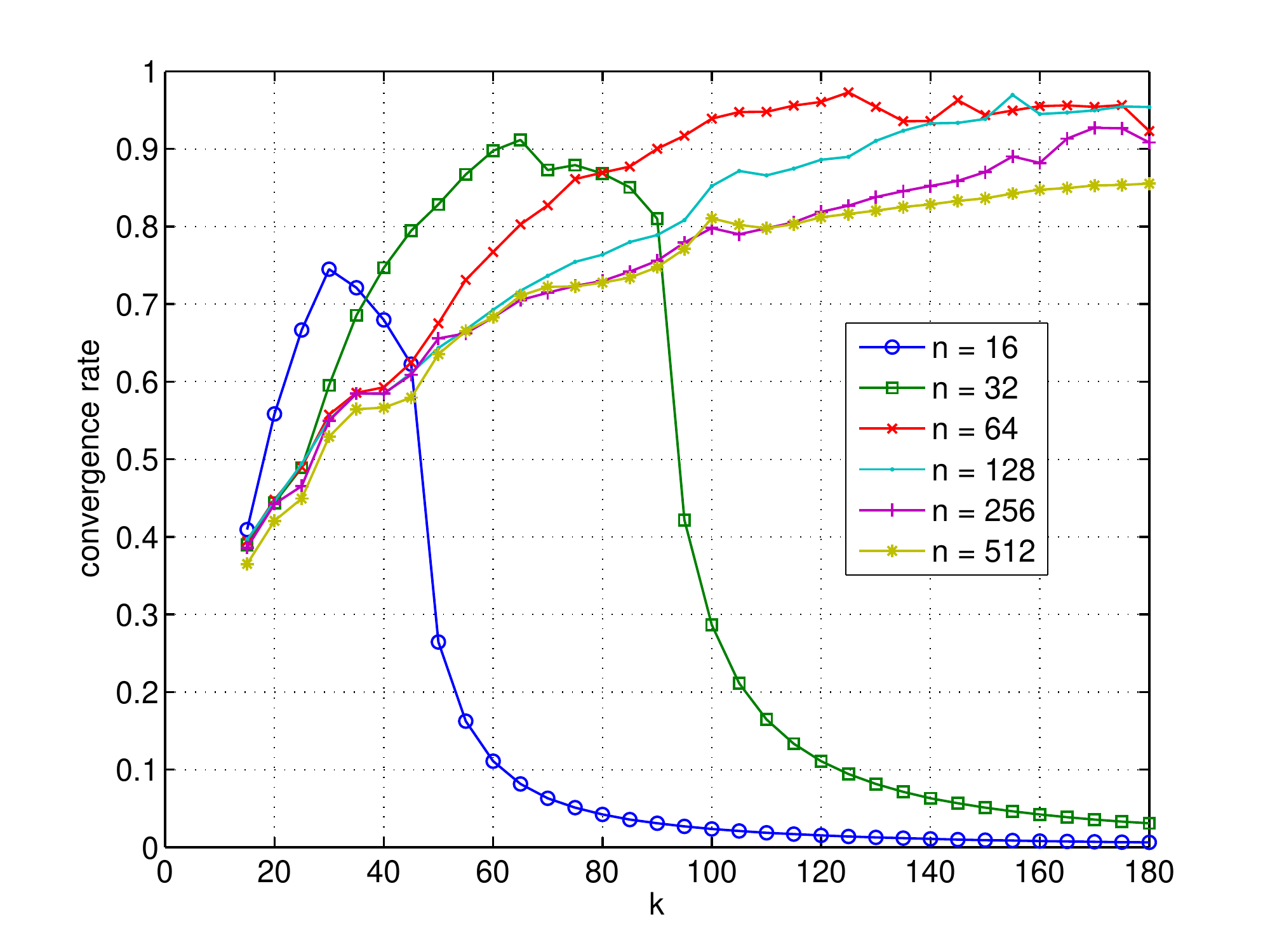}\label{fig:krylov_rate_gmres_exact}
}\\
\subfloat[The number of GMRES iterations.]{
\includegraphics[width=10cm]{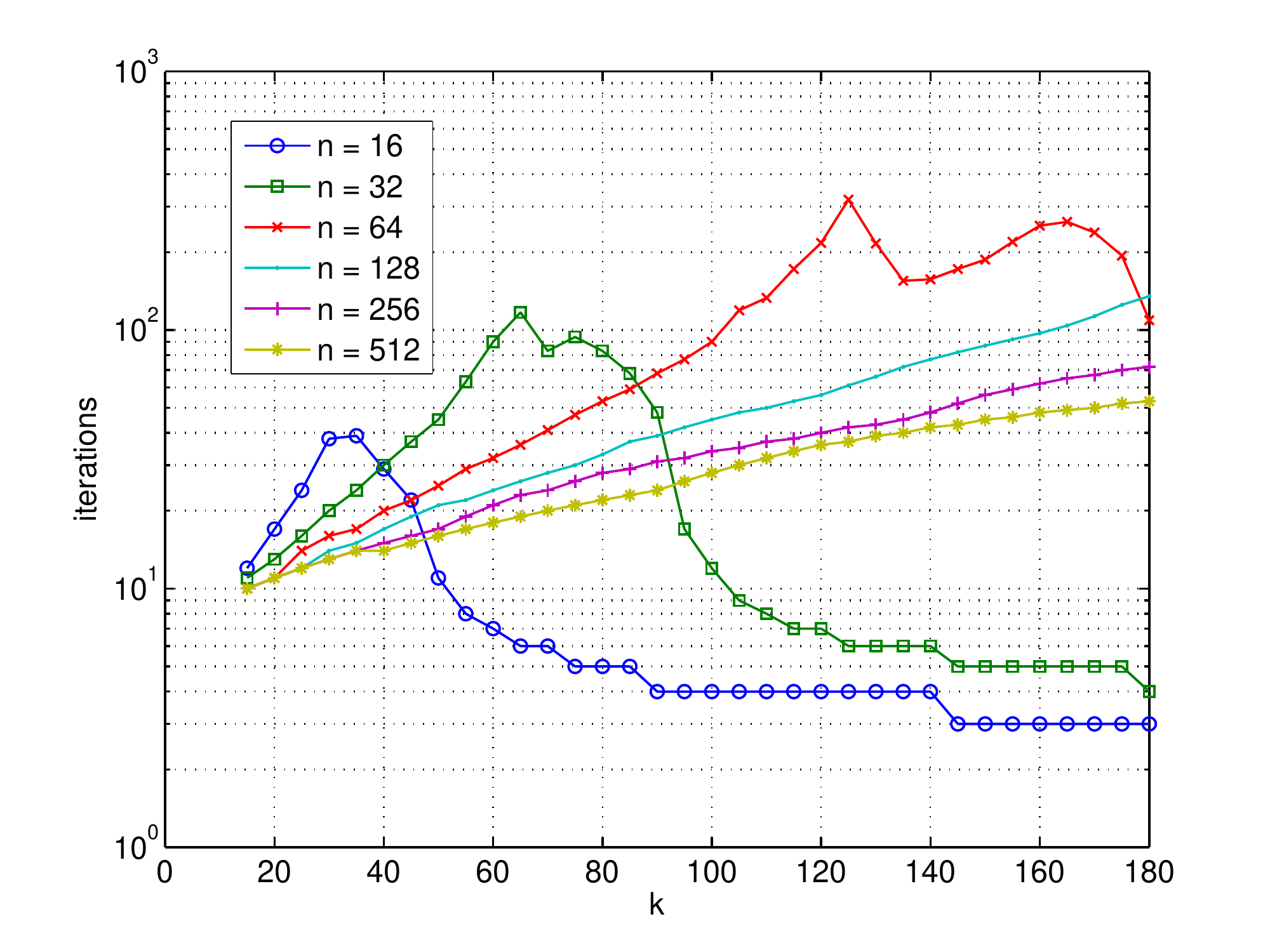}\label{fig:krylov_iter_gmres_exact}
}
\caption{Convergence results of GMRES as a function of the wave number $k$ for the homogeneous 2D Helmholtz problem. The preconditioner is exactly inverted. The lines represent different interior grid sizes $16\times16$, $32\times32$,\ldots, $512\times512$. The convergence patterns resemble those in Figure~\ref{fig:krylov_fgmres} for the FGMRES case where the inverse of the preconditioning matrix is approximated with one V(1,1)-cycle. The discussed effects are more pronounced.}
\label{fig:krylov_gmres_exact}
\end{center}
\end{figure}
Only for the two smallest grid sizes this behavior lies completely
within the tested range of wave numbers $k$. Indeed, in
Table~\ref{tab:criticalk} the critical wave numbers are
$k_1=\frac{2}{h}=2n$ and $k_2=\frac{2\sqrt{2}}{h}=2\sqrt{2}n$
are listed for the different grid sizes in the experiments. The first
critical wave number $k_1$ has the worst performance because the
spectrum of $H_h$ reaches its most extreme indefiniteness, with the
pitchfork perfectly spread over the negative and positive real part of
the lower half of the complex plane. For larger $k$ the spectrum tends
more to negative definiteness. The second peak is right before it
turns completely negative definite in the second critical wave number
$k_2$ after which the convergence rate obviously drops drastically.

\begin{table}
\centering
\begin{tabular}{c c c c c c c c c}
\hline
  $n$   & 16 & 32 & 64 & 128 & 256 & 512 & 1024 & 2048 \\
 \hline
  $k_1$ & 32 & 64 & 128 & 256 & 512 & 1048 & 2048 & 4096 \\
  $k_2$ & 45.3 & 90.5 & 181.0 & 362.0 & 724.1 & 1448.2 & 2896.3 & 5792.6\\
  \hline
   $k_b=\sqrt{|t_b|}$ & 12.8 & 13.8 & 20.3 & 26.2 & 35.7 & 40.6 & 50.0 & 53.1 \\
\hline
\end{tabular}
 \caption{The critical wave numbers $k_1$ and $k_2$ for the discrete Helmholtz problem with constant $k$ for different interior grid sizes $n$. For these values of $k$ the preconditioned Krylov method reaches its worst performance as we see in Figures~\ref{fig:krylov_fgmres} and \ref{fig:krylov_gmres_exact}. The absolute value of the branch point indicates a wave number $k_b=\sqrt{|t_b|}$ that marks the end of an early plateau in the convergence behavior discussed in Section~\ref{ssec:deviation}.}
 \label{tab:criticalk}
\end{table}

For a realistic physical solution the grid size should be large enough
in order to represent the wave accurately. Higher wave numbers $k$
require finer meshes \cite{babuska2000pollution}. This means that in practical
circumstances only the region on the curve long before the first peak
is important, where the spectrum is only slightly negative definite.
In that region we still profit from the fact that for $k\rightarrow 0$ the eigenvalues of the preconditioner
accumulate into a single point (see Section~\ref{ssec:deviation}). There
we see a rise and the stagnation of the number of iterations into a
plateau with the branch point $t_b$ as indicator. However, on Figures~\ref{fig:krylov_fgmres} and \ref{fig:krylov_gmres_exact} the range of wave numbers $k$ is too wide to clearly uncover this initial effect as in Figure~\ref{fig:gmres_iterations}.

\section{Discussion and Conclusions}
In this paper we have analyzed the convergence rate of a multigrid
preconditioned Krylov solver for Helmholtz problems with absorbing
boundary conditions. The multigrid method inverts a preconditioner that is a Helmholtz operator discretized on a complex-valued grid rather than on a real grid. This preconditioner is comparable to the complex shifted Laplacian.  The multigrid method uses GMRES($3$) as a smoother at each level.

To understand the Krylov convergence, we have proposed a model problem
with a Dirichlet boundary condition on one side and an outgoing wave
boundary condition at the other. The outgoing boundary condition is
implemented with exterior complex scaling (ECS) that extends the domain with a
complex-valued contour.  This model problem is representative for
the implementation of absorbing boundary layers in various
applications such as ECS, which is often used in chemistry and physics or PMLs, which are frequently used in engineering.

We have analyzed this model both in a continuous and a discrete way. For the continuous problem we have found that the spectrum of the preconditioned operator lies on a curve in the complex plane which is bounded away from zero. This leads to an expected Krylov convergence rate that is bounded for all wave numbers. For small wave numbers the convergence rate is faster since the eigenvalues accumulate to a single point.

In the discrete problem the spectrum behaves similarly to the continuous problem for small wave numbers. However, for larger wave numbers the spectrum can deviate significantly. This finds it origin in the properties of the discrete Helmholtz operator that has a pitchfork in the spectrum, where only one of the arms of the pitchfork approximates the spectrum of the continuous operator. These deviations destroy the nice convergence expectations given by the continuous operator. The distance to the origin of the predicted point $t_b$ where the spectrum bifurcates grows only very slowly as a function of the number of grid points. Numerical experiments show that also varying wave numbers result in similar pitchfork shaped spectra with branching points that can still be fairly well estimated using Proposition~\ref{prop:branchpoint}.

As a rule of thumb the Krylov convergence is bounded when $k^2$ is smaller than the absolute value of the branch point $t_b$. There we can
expect a bounded convergence rate. For wave numbers $k$ larger than this branch point the number of iterations rises until $k^2\approx 4/h^2$, where the number of
iterations is maximal.  We have a second but milder peak at $k^2\approx 8/h^2$. From then on the spectrum is negative definite and the convergence rate drops rapidly and only a few iterations are required to solve the system.

We conclude that there is no overall $k$-independent convergence rate, yet the number of iterations diminishes as the number of grid points is increased. To further improve the convergence rate of the iterative method it is possible to engineer the parameters of the absorbing layer as a
function of the wave number.  For large wave numbers $k$ we do not need
a large ECS grid or a large rotation angle to absorb the wave.
Adapting the parameters can reduce the number of iterations to solve
the problem.

Although the current analysis is for constant wave numbers $k$, we
believe many results will still be valid when the wave number varies
over space.  Indeed, a space-dependent $k(x)$ will only affect the
smoothest eigenvalues while the extreme values that determine the
diverging behavior depend on the grid distance and remain the same.

There still remain important challenges in the development of
an efficient solver for the Helmholtz problem with space-dependent wave
numbers based on complex stretched domains.  In numerical experiments
with strongly space-dependent wave numbers that allow evanescent waves we have
seen serious deteriorations of the convergence rate
\cite{polynomialsmoother}.  This is caused by the multigrid coarse grid
correction on levels too coarse to resolve the evanascent waves.  This
is a subject of future research.

\section*{Acknowledgement}
This research has been funded by the \textit{Fonds voor Wetenschappelijk Onderzoek (FWO)} by the project G.0174.08 and \textit{Krediet aan navorser} 1.5.145.10.\\
\br{The authors would like to thank Hisham bin Zubair for sharing a multigrid implementation.}

\bibliographystyle{wileyj}
\bibliography{bib_revision_arxiv}

\end{document}